\documentclass[a4paper]{amsart}

\usepackage{amssymb}
\usepackage{amsmath}
\usepackage{dsfont}
\usepackage{amsthm}

\usepackage{hyperref}

\newcommand{\BU}[1]{\boldsymbol{\mathrm{#1}}}
\newcommand{\U}[1]{\mathrm{#1}}
\newcommand{\B}[1]{\boldsymbol{#1}}
\DeclareMathOperator*{\var}{var}

\DeclareMathOperator*{\detcov}{detCov}
\DeclareMathOperator*{\cov}{Cov}
\DeclareMathOperator*{\vol}{Vol}

\theoremstyle{definition}
\newtheorem{thm}{Theorem}[section]
\newtheorem{prop}[thm]{Proposition}
\newtheorem{dfn}{Definition}[section]
\newtheorem{remark}{Remark}[section]
\newtheorem{cor}{Corollary}[section]
\newtheorem{lem}[thm]{Lemma}

\begin{document}

\title[Probability tails of local times]{On the probability distribution of the local times of diagonally operator-self-similar Gaussian fields with stationary increments}

\author[K. Kalbasi]{Kamran Kalbasi}
\thanks{Institute of Mathematics, EPFL (Swiss Federal Institute of Technology Lausanne)}

\author[T. Mountford]{Thomas S. Mountford}

\address{Institute of Mathematics, EPFL (Swiss Federal Institute of Technology Lausanne)}


\begin{abstract}
In this paper we study the local times of vector-valued Gaussian fields that are `diagonally operator-self-similar' and whose increments are stationary. Denoting the local time of such a Gaussian field around the spatial origin and over the temporal unit hypercube by $Z$, we show that there exists $\lambda\in(0,1)$ such that under some quite weak conditions, $\lim_{n\rightarrow +\infty}\frac{\sqrt[n]{\mathbb{E}(Z^n)}}{n^\lambda}$ and
$\lim_{x\rightarrow +\infty}\frac{-\log \mathbb{P}(Z>x)}{x^{\frac{1}{\lambda}}}$ both exist and are strictly positive (possibly $+\infty$).
Moreover, we show that if the underlying Gaussian field is `strongly locally nondeterministic', the above limits will be finite as well. These results are then applied to establish similar statements for the intersection local times of diagonally operator-self-similar Gaussian fields with stationary increments.
\end{abstract}

\maketitle

%
%
%
%
%
%

\section{Introduction}
Let $(\Omega, \mathcal{F}, \mathbb{P})$ be a probability space, and
$\B{X}_{\B{t}}:=\bigl(X^1_{\B{t}}, \cdots, X^d_{\B{t}}\bigr)$, $\B{t}\in\mathbb{R}^N$ be an $N$-parameter $d$-dimensional centered Gaussian field on $(\Omega, \mathcal{F}, \mathbb{P})$, i.e. each component $X^i_{\B{t}}$ is a real-valued zero-mean Gaussian field indexed by $\mathbb{R}^N$. We call such a random field a centered Gaussian $(N,d)$-field.

We denote by $\mathbb{R}_+$,  $\mathbb{N}$ and $\mathbb{Q}$ respectively the sets of strictly positive real numbers ($>0$), strictly positive integers ($\geq 1$), and finally the rational numbers. Evidently $\mathbb{R}_{\geq0}$ denotes the real numbers that are positive or zero. We denote the space of matrices of size $m\times n$ with real entries by $\mathbb{R}^{n\times m}$.
For any two same-sized vectors $\B{u}=(u_1, \cdots, u_n)$ and $\B{v}=(v_1, \cdots, v_n)$ in $\mathbb{R}^n$, $\B{u}\circ\B{v}$ denotes their Schur product, i.e.,  the vector $\B{u}\circ\B{v}:=(u_1v_1, \cdots, u_nv_n)$. For any square matrix $\BU{Y}$, we denote its trace (i.e. the sum of all its diagonal entries) by $\textrm{tr}(\BU{Y})$. For any matrix $\BU{Y}$, we denote its transpose by $\BU{Y}^\dagger$. For any matrices $\BU{A_1}$, $\BU{A_2}$, ..., $\BU{A_n}$, we define $\textrm{diag}(\BU{A_1}, \BU{A_2}, \cdots, \BU{A_n})$ as the block diagonal matrix that has matrices $\BU{A_1}$, $\BU{A_2}$, ..., $\BU{A_n}$ on its diagonal (respecting the order) and is zero elsewhere.

For any $\B{p}\in\mathbb{R}^N$ and $\B{T}\in\mathbb{R}_+^N$, let $\mathcal{C}({\B{p}},\B{T})$ denote $\B{p}+\prod_{i=1}^N[0,T_i]$, i.e., the $N$-dimensional cube of side lengths equal to $\{T_i\}_{i=1}^N$ and based at point $\B{p}$. We also denote $[0,1]$ by $\mathcal{I}$.

For any measurable subset $\mathcal{B}\subset\mathbb{R}^d$, we denote its Lebesgue measure by $\vol(\mathcal{B})$. For any subset $\mathcal{A}$ of an arbitrary set $\mathcal{X}$, we denote its indicator function by  $\mathbf{1}_{\{\mathcal{A}\}}$, i.e.
$$
\mathbf{1}_{\{\mathcal{A}\}}(x):=
\begin{cases}
1 \;,\quad \text{for} \;x\in \mathcal{A}\\
0 \;,\quad \text{for} \;x\not\in \mathcal{A}.
\end{cases}
$$
For any $k$-dimensional Gaussian random vector $\B{Y}=(Y_1, \cdots, Y_k)$, we denote the determinant of its covariance matrix  by $\detcov[\B{Y}]$; in other words $$\detcov[\B{Y}]:=\det \bigl[\mathbb{E} \bigl( \B{Y}\, \B{Y}^\dagger\bigr)-\mathbb{E}(\B{Y})\mathbb{E}(\B{Y}^\dagger)\bigr],$$ where $\B{Y}$ is regarded as a $k\times 1$ matrix. For any finite family of vectors $\B{y}_i=(y_1^i,\cdots,y_k^i)$, $i=1,\cdots,n$, we call the following vector as their adjoined vector:
$$
[y_1^1,\cdots,y_k^1, y_1^2,\cdots,y_k^2, \cdots, y_1^n,\cdots,y_k^n],
$$
and we denote it by $[\B{y}_1, \cdots, \B{y}_n]$.

Once a centered Gaussian $(N,d)$-field $\B{X}$ is fixed, for any positive integer $k$ and any $\B{t}_1, \cdots, \B{t}_k\in \mathbb{R}^N$ we define
\begin{equation}\label{K_n Definition}
\U{K}_n^{\B{X}}(\B{t}_1, \cdots, \B{t}_n):= (2\pi)^{-\frac{nd}{2}}\bigl(\det \cov \bigl[\B{X}_{\B{t}_1}, \cdots, \B{X}_{\B{t}_n}]\bigr)^{-\frac{1}{2}}.
\end{equation}

We use the following definition of local times which provides a pointwise characterisation. The more common definition of local times as the Radon-Nikodym derivative of the occupation measure of a random process (if it is absolutely continuous), only provides an almost-sure characterisation of the occupation density. In section \ref{Two different definitions of local times}, we will see more on this and the link between the two definitions.


\begin{dfn}
Let $\{\B{X}_{\B{t}}\}_{\B{t}}$ be a random field on $\mathbb{R}^N$ with values in $\mathbb{R}^d$. We define the local time of $\B{X}$ at $\B{x}\in\mathbb{R}^d$ and over the cube $\mathcal{C}(\B{p},\B{T})$  as the following limit (if it exists)
$$
L_{\B{x}}(\B{X};\mathcal{C}(\B{p},\B{T})):=
\lim_{\varepsilon\rightarrow 0}
\int_{\mathcal{C}(\B{p},\B{T})}\frac{1}{\vol(\mathcal{B}_\varepsilon(\B{x}))}
\mathbf{1}_{\{\|\B{X}_{\B{t}}-\B{x}\|<\varepsilon\}}(\B{t})\mathrm{d}\B{t},
$$
where $\|\cdot\|$ is an arbitrary norm on $\mathbb{R}^d$, and $\mathcal{B}_\varepsilon(\B{x}):=\{\B{y}\in\mathbb{R}^d ; \|\B{x}-\B{y}\|<\varepsilon\}$.
\end{dfn}


We are interested in the tail-decay behavior of the probability distribution of $L_{\B{0}}(\B{X};\mathcal{I}^N)$.
The first work in this direction goes back to \cite{KasahKonoOgawa99}. They consider a one-parameter one-dimensional ($N=d=1$) Gaussian process $X(t)$ with stationary increments satisfying the local nondeterminism condition \cite{Xiao08}. Moreover, defining $\sigma^2(t)=\mathbb{E}\bigl[(X_t-X_s)^2\bigr]$, they assume that $\sigma(t)$ is continuous and strictly increasing on the interval $[0,1]$, that $\frac{1}{\sigma(t)}$ is integrable over $\mathcal{I}=[0,1]$, and finally $\sigma(t)$ varies regularly at $0$ with some exponent $0<H<1$, i.e., $\lim_{t\rightarrow0}\frac{\sigma(\omega t)}{\sigma(t)}=\omega^H$ for every $\omega>0$. In fact this latter condition is a gauge for asymptotic self-similarity near the origin. Under these conditions they show that the local times of $X(\cdot)$ exist, and moreover
$$
0< \liminf_{x\rightarrow +\infty}\frac{-\log\mathbb{P}[L_{0}(X,\mathcal{I})>x]}{\sigma^{-1}(\frac{1}{x})}\leq \limsup_{x\rightarrow +\infty}\frac{-\log\mathbb{P}[L_{0}(X,\mathcal{I})>x]}{\sigma^{-1}(\frac{1}{x})} <+\infty.
$$
When $\sigma(t)=t^H$, which corresponds to the fractional Brownian motion of Hurst parameter $H$, the exponential decay rate $\sigma^{-1}(\frac{1}{x})$ equals $x^{\frac{1}{H}}$.

More recently, \cite{ChenLiRosinskyShao11} considers the one-parameter d-dimensional fractional Brownian motion $\B{B}^H(t)=(B^H_1(t), \cdots,B^H_d(t))$ and also d-dimensional fractional Riemann-Liouville process $\B{W}^H(t)=(W^H_1(t), \cdots,W^H_d(t))$ where $\{B^H_i\}_{i=1}^d$ ($\{W^H_i\}_{i=1}^d$) are $d$ independent copies of a fractional Brownian motion (fractional Riemann-Liouville process) with Hurst parameter $H$. They show that the following limits exist
$$
\lim_{x\rightarrow +\infty}
x^{-\frac{1}{dH}}\log\mathbb{P}[L_{\B{0}}(\B{B}^H,\mathcal{I})>x]
\qquad \text{and}\qquad
\lim_{x\rightarrow +\infty}
x^{-\frac{1}{dH}}\log\mathbb{P}[L_{\B{0}}(\B{W}^H,\mathcal{I})>x].
$$

We will prove the existence of this exponential tail-decay limit for the class of Gaussian fields that have stationary increments (Property $\mathfrak{A}_2$ below) and are `diagonally self-similar' as defined in Property $\mathfrak{A}_3$ below.

Throughout the paper we assume that the random field $\B{X}$ has both of the following two properties ($\mathfrak{A}_0$ and $\mathfrak{A}_1$).

\noindent
\textbf{Property $\mathfrak{A}_0$: }
There exists a positive constant $c_0>0$ such that $\var(X_{\B{t}}^i)\leq c_0$ for every $\B{t}\in[0,1]^N$ and $i=1,\cdots,d$.\\
As we do not assume any kind of continuity of $\B{X}$ or its covariance matrix, the boundedness of its variance (Property $\mathfrak{A}_0$), seems inevitable.

\noindent
\textbf{Property $\mathfrak{A}_1$: }
The (N,d)-Gaussian field $\B{X}$ has the property that for any positive integer $n$,
the expression 
$\U{K}_n^{\B{X}}(\B{t}_1, \cdots, \B{t}_n)$
is integrable over $\bigl(\mathcal{I}^N\bigr)^{n}$.\\
Property $\mathfrak{A}_1$ guarantees the existence of the local times at every point, i.e. $L_{\B{x}}(\B{X};\mathcal{I}^N)$, and the finiteness of all their moments, see Proposition \ref{existence and approximation}. In fact, $\mathfrak{A}_1$ is the weakest-known sufficient condition for the existence of local time at the origin and the finiteness of all its moments.

Next we have the following two properties that form our main framework.

\noindent
\textbf{Property $\mathfrak{A}_2$(Stationary Increments): }The random field $\B{X}$ is zero at the origin and has stationary increments, i.e., for any $\B{p}\in\mathbb{R}^N$ we have the following equality for every $\B{s}, \B{t}\in\mathbb{R}^N$ and $i,j\in\{1, \cdots, d\}$
$$
\mathbb{E}\bigl[X^{(i)}_{\B{s}}X^{(j)}_{\B{t}}\bigr]
=\mathbb{E}\bigl[(X^{(i)}_{\B{s+p}}-
X^{(i)}_{\B{p}})(X^{(j)}_{\B{t+p}}-X^{(j)}_{\B{p}})\bigr].
$$
%


\noindent
\textbf{Property $\mathfrak{A}_3$(Diagonal Self-Similarity): }
There exist a vector $\B{\alpha}=(\alpha_1, \cdots, \alpha_N)\in \mathbb{R}_+^N$ and a matrix $\BU{H}\in\mathbb{R}^{d\times d}$ with positive trace ($\textrm{tr}(\BU{H})>0$)
such that for every $\omega>0$ we have
\begin{equation}\label{diagonal self-similarity equality in distribution}
\B{X}_{\B{t}\circ\omega^{\B{\alpha}}}
\overset{d}{=}\omega^{\BU{H}} \B{X}_{\B{t}}
\;;\quad \forall \omega\in\mathbb{R}_+,
\end{equation}
where $\omega^{\B{\alpha}}:=(\omega^{\alpha_1}, \cdots, \omega^{\alpha_1})$, the values of $\B{X}_{\B{t}}$ are considered as $d\times 1$ matrices, $\overset{d}{=}$ means equality in finite dimensional distributions for the two random fields, and
$\omega^{\BU{H}}$ denotes matrix exponential with the usual definition, i.e.
$$
\omega^{\BU{H}}:=e^{\ln(\omega)\,\BU{H}}
=\sum_{i=0}^{\infty}\frac{\bigl(\ln(\omega) \,\BU{H}\bigr)^n}{n!}
\;;\quad \forall \omega\in \mathbb{R}_+.
$$

\begin{remark}
This definition is a special case of the more general concept of what is called operator-self-similar random fields, e.g. studied in \cite{LiXiao2011}. In the general case, the vectors $\B{\alpha}$ and hence $\omega^{\B{\alpha}}$ are replaced by a matrix $\BU{E}$ and its matrix exponential $\omega^{\BU{E}}$, respectively. Evidently in this more general setting, the Shur product $\omega^{\B{\alpha}}\circ \B{t}$ should be replaced by the usual matrix multiplication $\omega^{\BU{E}} \B{t}$. This justifies us calling Property $\mathfrak{A}_3$ as `diagonal' self-similarity.
\end{remark}



\begin{remark}\label{diagonal self-similarity covariance formulation}
For zero-mean Gaussian fields, Equation \eqref{diagonal self-similarity equality in distribution} in Property $\mathfrak{A}_3$ is equivalent to the following equation
$$
\mathbb{E}\bigl(\BU{X}_{\B{s}\circ\omega^{\B{\alpha}}} \BU{X}_{\B{t}\circ\omega^{\B{\alpha}}}^\dagger
\bigr)= \omega^{\BU{H}} \, \mathbb{E}\bigl(\BU{X}_{\B{s}} \BU{X}_{\B{t}}^\dagger\bigr)\omega^{\BU{H}^\dagger},
\quad \forall \B{s}, \B{t}\in\mathbb{R}^N,
$$
where $\B{X}_{\B{t}}$ is considered as a $d\times 1$ matrix as above, and $\BU{H}^\dagger$ denotes the transpose of matrix $\BU{H}$. For more on matrix exponential see e.g. \cite[ch.2]{Hall2003}.
\end{remark}

An important special case of Property $\mathfrak{A}_3$ is the following condition.

\noindent
\textbf{Property $\mathfrak{A}_3^\circ$(Two-sided Diagonal Self-Similarity): }
There exist $\B{\alpha}=(\alpha_1, \cdots, \alpha_N)\in \mathbb{R}_+^N$ and $(H_1, \cdots, H_d)\in \mathbb{R}_+^d$ such that for every $\omega>0$ we have
\begin{equation}\label{diagonal self-similarity covarience formula}
\mathbb{E}\bigl[X^{(i)}_{\B{s}\circ\omega^{\B{\alpha}}}
X^{(j)}_{\B{t}\circ\omega^{\B{\alpha}}}\bigr]= \omega^{H_i+H_j} \, \mathbb{E}\bigl[X^{(i)}_{\B{s}}X^{(j)}_{\B{t}}\bigr],
\quad \forall i,j\in\{1,\cdots,d\} \;,\, \forall \B{s}, \B{t}\in\mathbb{R}^N,
\end{equation}
where $\omega^{\B{\alpha}}:=(\omega^{\alpha_1}, \cdots, \omega^{\alpha_1})$.

\begin{remark}
For a zero-mean Gaussian field $\B{X}_{\B{t}}$, Property $\mathfrak{A}_3^\circ$ is satisfied if and only if Property  $\mathfrak{A}_3$ is satisfied with $\BU{H}=\textrm{diag}(H_1, H_2, \cdots, H_d)$, i.e., the diagonal matrix whose diagonal entries are $\omega^{H_1}$,..., $\omega^{H_d}$ (respecting the order) and is zero elsewhere. This is true because we have
$$
\omega^{\textrm{diag}(H_1, H_2, \cdots, H_d)}=\textrm{diag}(\omega^{H_1},\cdots, \omega^{H_d}).
$$
\end{remark}

\begin{remark}
A very important random field that satisfies both Properties $\mathfrak{A}_2$ and $\mathfrak{A}_3^\circ$ is the multi-parameter fractional Brownian motion, i.e. the centered Gaussian field with stationary increments characterised by $\mathbb{E}[(X_{\B{s}}-X_{\B{t}})^2]=|\B{s-t}|^{2H}$ for every $\B{s}, \B{t}\in\mathbb{R}^N$, where $H\in(0,1]$ is the Hurst parameter of the Gaussian field. Furthermore, the centered Gaussian $(N,d)$-field consisting of $d$ independent multi-parameter fractional Brownian motions each with its own Hurst parameter $H_i$ satisfies also $\mathfrak{A}_2$ and $\mathfrak{A}_3$, hence falls in the scope of this paper as well.
\end{remark}

\begin{remark}
Let $c_1, \cdots, c_N\in \mathbb{R}_+$, $p_1, \cdots, p_N\in (0,2]$ and $H\in(0,1]$. Consider the $(N,1)$-Gaussian field that we call `anisotropic fractional Brownian motion', i.e., the $\mathbb{R}^N$-indexed centered Gaussian field $X_{\B{s}}$ with stationary increments given by
$$
\mathbb{E}[(X_{\B{s}}-X_{\B{t}})^2]=\phi(\B{s}-\B{t}),
$$
where
$$
\phi(\B{s})=\bigl(\Sigma_{i=1}^N c_i|s_i|^{p_i}\bigr)^{2H}.
$$
This Gaussian field satisfies both Properties $\mathfrak{A}_2$, and $\mathfrak{A}_3^\circ$ with $\tilde{\B{\alpha}}:=(\frac{1}{p_1}, \cdots, \frac{1}{p_N})$ and $H_1:=H$.

For other interesting examples of operator-self-similar random processes and fields with stationary increments, see e.g. \cite{MaejimaMason1994}.

In Section \ref{MainResults} we gather all the main results of this paper. In Section \ref{Two different definitions of local times}, we discuss the pointwise versus functional definitions of local times which are relevant to our work. In Section \ref{Formulation in moments growth rate} we state the relation between the exponential decay rate of the probability tail of local times and the exponential growth rate of their moments. Sections \ref{Section on Upper bounds} and \ref{Existence of the limit} contain the technical proofs.

\section{Main Results}\label{MainResults}
In this section we give some technical definitions and state our results. The proofs will come in the subsequent sections.

For every $\B{\alpha}=(\alpha_1, \alpha_2, \cdots, \alpha_N)\in \mathbb{R}_+^N$,
we define the $\B{\alpha}$-length as follows
\begin{equation}\label{alpha distance}
\|\B{t}\|_{\B{\alpha}}:=\sum_{i=1}^{N}|t_i|^{1/\alpha_i}\;:\quad \forall \B{t}=(t_1, \cdots, t_N)\in \mathbb{R}^N
\end{equation}
It is evident that $\|\B{t}\|_{\B{\alpha}}$ defines a translation invariant topology on $\mathbb{R}_+^N$. Moreover, if $\forall i=1, \cdots, N : \alpha_i\geq1$ then $\|\B{t}\|_{\B{\alpha}}$ defines a translation invariant metric on $\mathbb{R}_+^N$ which we call the $\B{\alpha}$-distance. Nevertheless, it is not a norm except for the special case where all the exponents are equal to $1$.

We introduce the following definition which generalizes the idea of Strong Local Nondeterminism to vector-valued Gaussian fields.

\begin{dfn}[Strong Local Nondeterminism]\label{local nondeterminism definition}
We call a centered Gaussian $(N,d)$-field $\B{X}_{\B{t}}$ strongly locally nondeterministic over a cube $\mathcal{J}\subseteq\mathbb{R}^N$ with scaling vector $\B{\xi}:=(\xi_1, \xi_2, \cdots, \xi_N)\in\mathbb{R}_+^N$ if there exist constants $H>0$ and $C>0$
such that for any positive integer $n$, and any arbitrary vectors $\B{u}, \B{t}_1, \cdots, \B{t}_n\in \mathcal{J}$, we have
$$
\detcov[\B{X}_{\B{u}}|\B{X}_{\B{t}_1}, \B{X}_{\B{t}_2},\cdots, \B{X}_{\B{t}_n}]\geq C \min_{0\leq i\leq n}
\|\B{u}-\B{t}_i\|_{\B{\alpha}}^{2H},
$$
where $\B{t}_0:=\B{0}$, the expression
$\detcov[\B{X}_{\B{u}}|\B{X}_{\B{t}_1}, \B{X}_{\B{t}_2},\cdots, \B{X}_{\B{t}_n}]$ denotes the determinant of the conditional covariance matrix of the random vector $\B{X}_{\B{u}}$ conditioned on all the random vectors $\B{X}_{\B{t}_1}, \B{X}_{\B{t}_2},\cdots, \B{X}_{\B{t}_n}$, and finally, $\B{\alpha}=(\alpha_1, \alpha_2, \cdots, \alpha_N):=H \B{\xi}$.
\end{dfn}

\begin{remark}
The reason why only the normalized vector $\B{\xi}=\frac{1}{H}\B{\alpha}$ is relevant, is due to the fact that for any $p>0$ there exist positive constants $c_1, c_2>0$ such that for every $\B{x}=(x_1, \cdots, x_N)\in \mathbb{R}^N$
$$
c_1(\sum_{i=1}^{N}|x_i|^{1/\alpha_i})^{2H}\leq
(\sum_{i=1}^{N}|x_i|^{p/{\alpha_i}})^{^\frac{2H}{p}}\leq c_2(\sum_{i=1}^{N}|x_i|^{1/\alpha_i})^{2H}.
$$
In fact we have the following proposition.
\end{remark}
\begin{prop}\label{Equivalence of self-similar functions}
Let $f:\mathbb{R}^N\rightarrow\mathbb{R}^{\geq 0}$ be a continuous function such that $f(\B{x})=0$ if and only if $\B{x}=\B{0}$, and for some vector $\B{\alpha}=(\alpha_1, \alpha_2, \cdots, \alpha_N)\in\mathbb{R}_+^N$ and $H>0$, we have $f(\B{x}\circ \omega^{\B{\alpha}})=\omega^H f(\B{x})$ for every $\B{x}\in \mathbb{R}^N$ and $\omega>0$. Then there exist constants $c_1, c_2>0$ such that for every $\B{x}=(x_1, \cdots, x_N)\in \mathbb{R}^N$
$$
c_1(\sum_{i=1}^{N}|x_i|^{1/\alpha_i})^{H}\leq
f(\B{x})\leq c_2(\sum_{i=1}^{N}|x_i|^{1/\alpha_i})^{H}.
$$
\end{prop}
\begin{proof}
In Section \ref{Section on Upper bounds}.
\end{proof}
\begin{remark}
Let $\B{X}_{\B{t}}$ be a diagonally self-similar centered Gaussian $(N,d)$-field that satisfies Property $\mathfrak{A}_3$ with matrix $\BU{H}\in \mathbb{R}^{d\times d}$ and vector $\B{\alpha}=(\alpha_1, \cdots, \alpha_N)\in \mathbb{R}_+^N$. If $\B{X}_{\B{t}}$ is strongly locally nondeterministic with the scaling vector $\B{\xi}:=(\xi_1, \xi_2, \cdots, \xi_N)$, then it is easy to verify that $\B{\xi}=\frac{1}{\textrm{tr}(\B{H})}\B{\alpha}$. In other words, for diagonally self-similar centered Gaussian $(N,d)$-fields, the strong local nondeterminism can be satisfied only with a unique scaling vector.
\end{remark}


\begin{prop}\label{parameters restrictions}
Let $\B{X}_{\B{t}}$ be a diagonally self-similar centered Gaussian $(N,d)$-field with stationary increments, i.e. it satisfies Properties $\mathfrak{A}_0$, $\mathfrak{A}_2$, and $\mathfrak{A}_3$ with some matrix $\BU{H}\in \mathbb{R}^{d\times d}$ and vector $(\alpha_1, \cdots, \alpha_N)\in \mathbb{R}_+^N$. Let $\beta$ be a positive real number. If the kernel
$\big(\U{K}_n^{\B{X}}\bigr)^{\beta}$ is integrable over the cube $\bigl(\mathcal{I}^N\bigr)^{n}$ for some integer $n$, then the following inequality has to hold true
$$
\sum_{i=1}^{N}\alpha_i > \beta \, \textrm{tr}(\BU{H}).
$$
\end{prop}
\begin{proof}
In Section \ref{Section on Upper bounds}.
\end{proof}

\begin{lem}\label{local nondeterminism theorem}
Let $\B{X}_{\B{t}}$ be a centered Gaussian $(N,d)$-field which is strongly locally nondeterministic over $\mathcal{I}^N$ with
scaling vector $\B{\xi}=(\xi_1, \xi_2, \cdots, \xi_N)$ and constant $C_0$.
Then for any positive real number $\beta$ such that
$\beta <\sum_{i=1}^{N}\xi_i$
, and any positive integer $n$, the kernel $\big(\U{K}_n^{\B{X}}\bigr)^{\beta}$ is integrable over the cube $\bigl(\mathcal{I}^N\bigr)^n$, and
$$
\int_{(\mathcal{I}^N)^n}\big(\U{K}_n^{\B{X}}(\B{t}_1, \cdots, \B{t}_n)\bigr)^{\beta}\mathrm{d}\B{t}_1\cdots\mathrm{d}\B{t}_n\leq
c^n \, (n!)^{\frac{\beta}{\sum_{i=1}^{N}\xi_i}}
$$
where $c$ is a constant that depends only on $C_0$, $N$, $\B{\alpha}$, $\beta$, $H$ and $d$.
\end{lem}
\begin{proof}
In Section \ref{Section on Upper bounds}.
\end{proof}

\begin{thm}
Let $\B{X}_{\B{t}}$ be a centered Gaussian $(N,d)$-field that is strongly locally nondeterministic over $\mathcal{I}^N$ with
scaling  vector $\B{\xi}=(\xi_1, \xi_2, \cdots, \xi_N)\in \mathbb{R}_+^N$ such that $1<\sum_{i=1}^{N}\xi_i$. Then
the local times $Z_{\B{x}}:=L_{\B{x}}(\B{X},\mathcal{I}^N)$ of the random field $\B{X}_{\B{t}}$ exist at every point $\B{x}\in \mathbb{R}^d$, and
$$
\mathbb{E}(Z_{\B{x}}^n)\leq c^n \, (n!)^{\frac{ 1}{\sum_{i=1}^{N}\xi_i}},
$$
for some constant $c$ which does not depend on $n$.
\end{thm}
\begin{proof}
It is immediate from Lemma \ref{local nondeterminism theorem}.
\end{proof}

\begin{dfn}
Let $\{\B{X}_{k}(\B{t}_k)\,:\,\B{t}_k\in \mathbb{R}^{N_k}\}_{k=1}^m$ be a family of $m$ independent Gaussian fields such that for every $k=1, \cdots, m$, the random field $\B{X}_{k}(\B{t})$
is a centered Gaussian $(N_k,d)$-field. We define their ($m$-fold) intersection local time around the origin and over $\mathcal{I}$  as the local time of the following $(\sum_{k=1}^m N_k,(m-1)d)$-field at $\B{0}$ and over the cube $\mathcal{I}^{\sum_{k=1}^m N_k}$ (if it exists)
$$
\bigl(\B{X}_1(\B{t_1})-\B{X}_2(\B{t_2}), \B{X}_2(\B{t_2})-\B{X}_3(\B{t_3}),\cdots, \B{X}_{m-1}(\B{t_{m-1}})-\B{X}_m(\B{t_m})\bigr).
$$
\end{dfn}

\begin{thm}
\label{Upper bound for intersection local times}
Let $m$ be a positive integer, and $\{\B{X}_{k}(\B{t}_k)\,:\,\B{t}_k\in \mathbb{R}^{N_k}\}_{k=1}^m$ be a family of $m$ independent Gaussian fields such that for every $k=1, \cdots, m$, the random field $\B{X}_{k}(\B{t}_k)$ is a centered Gaussian $(N_k,d)$-field that is strongly locally nondeterministic with scaling vector $\B{\xi}_k=(\xi_{k,1}, \xi_{k,2}, \cdots, \xi_{k,N_k})\in \mathbb{R}_+^{N_k}$.
If $\sum_{k=1}^{m}\sum_{i=1}^{N_k}\xi_{k,i}>m-1$
, then the $m$-fold intersection local time of the family $\{\B{X}_{k}(\cdot)\}_{k=1}^m$ over the interval $\mathcal{I}$ exists. Moreover, denoting this intersection local time by $\mathfrak{I}_{\B{X}}$, and defining $\tilde{\xi}_k:=\sum_{i=1}^{N_k}\xi_{k,i}$, for any arbitrary sequence of positive numbers $q_1$, $q_2$, ..., $q_m$ such that $\sum_{k=1}^{m}q_k=m-1$, and such that $0\leq q_k\leq 1$ and $q_k<\tilde{\xi}_k$ (for every $k=1, \cdots, m$), we have
$$
\mathbb{E}\bigl((\mathfrak{I}_{\B{X}})^n\bigr)\leq c^n \,
(n!)^{\sum_{k=1}^{m} \frac{q_k}{\tilde{\xi}_k}},
$$
where $c$ is a positive constant that does not depend on $n$.\\
\end{thm}
\begin{proof}
In Section \ref{Section on Upper bounds}.
\end{proof}

\begin{thm}\label{main theorem on local times}
Suppose $\B{X}_{\B{t}}$ is a centered Gaussian $(N,d)$-field satisfying Properties $\mathfrak{A}_0$, $\mathfrak{A}_1$, $\mathfrak{A}_2$, and $\mathfrak{A}_3$ with some matrix $\BU{H}\in \mathbb{R}^{d\times d}$ and vector $(\alpha_1, \cdots, \alpha_N)\in \mathbb{R}_+^N$
such that $\alpha_i$'s are mutually rational, i.e., $\frac{\alpha_i}{\alpha_j}\in \mathbb{Q}$ for every $i$ and $j$. Then the following limits exist in $\mathbb{R}_+\bigcup\{+\infty\}$
, and are strictly positive
$$
\lim_{x\rightarrow +\infty}\frac{-\log \mathbb{P}(Z>x)}{x^{\frac{1}{\lambda}}}\quad \text{and}\quad \lim_{n\rightarrow +\infty}\frac{\sqrt[n]{\mathbb{E}(Z^n)}}{n^{\lambda}},
$$
where $Z:=L_{\B{0}}(\B{X},[0,1]^N)$ and $\lambda:=\frac{\textrm{tr}(\BU{H})}{\sum_{k=1}^{N}\alpha_k}$. Moreover, if $\B{X}_{\B{t}}$ is also strongly locally nondeterministic over $\mathcal{I}^N$,
then the above limits will be finite.
\end{thm}

\begin{proof}
In Section \ref{Existence of the limit}.
\end{proof}

\begin{remark}\label{Exceptional Case}
One should note that although Properties $\mathfrak{A}_0$, $\mathfrak{A}_1$, $\mathfrak{A}_2$, and $\mathfrak{A}_3$ guarantee the convergence of the sequence $\{\frac{\bigl(\mathbb{E}(Z^n)\bigr)^{\frac{1}{n}}}{n^{\lambda}}\}_n$, they do not imply the finiteness of the limit. Probably the simplest example would be the centered Gaussian (2,1)-field $X(s,t)$ ($s,t\in \mathbb{R}$)
characterized by $X(0,0)=0$ and $\mathbb{E}\bigl(X(s_1,t_1)-X(s_2,t_2)\bigr)^2=(t_1-s_1)^{2H}+(t_2-s_2)^{2H}$, where $H\in(0,1)$. It clearly satisfies all the properties $\mathfrak{A}_0$, $\mathfrak{A}_1$, $\mathfrak{A}_2$, and $\mathfrak{A}_3$ with the self-similarity scaling $(\B{\alpha},H)$, where $\B{\alpha}:=(1,1)$. So by Theorem \ref{main theorem on local times}, we know that $\{\frac{\sqrt[n]{\mathbb{E}(L_0^n)}}{n^{H/2}}\}_n$ converges, where $L_0$ is the local time of $X$ around the origin on the square $\mathcal{I}^2$. On the other hand, one can easily verify that $X$ is equivalent to $\{B_1(s)-B_2(t)\,:\, (s,t)\in \mathbb{R}^2\}$, where $B_1$ and $B_2$ are two independent fractional Brownian motions of Hurst parameter $H$. So $L_0$, i.e., the local time of $X$ around the origin, is the same as $\mathfrak{I}_B$, i.e., the intersection local time of two independent fractional Brownian motions with the same Hurst parameter. By \cite[Theorem 2.4.]{ChenLiRosinskyShao11}, we know that $\{\frac{\sqrt[n]{\mathbb{E}(\mathfrak{I}_B^n)}}{n^{H}}\}_n$ converges to a strictly positive finite constant. This shows that the right growth exponent of $\sqrt[n]{\mathbb{E}(L_0^n)}$ is $n^{H}$.
\end{remark}


\begin{cor}\label{limit theorem on intersection local times}
Let $m$ be a positive integer, and $\{\B{X}_{k}(\B{t}_k)\,:\,\B{t}_k\in \mathbb{R}^{N_k}\}_{k=1}^m$ be a family of $m$ independent Gaussian fields such that for every $k=1, \cdots, m$, the random field $\B{X}_{k}(\B{t})$ is a centered Gaussian $(N_k,d)$-field
satisfying Properties $\mathfrak{A}_0$, $\mathfrak{A}_2$, and $\mathfrak{A}_3$ with the self-similarity scaling matrix $\BU{H}\in \mathbb{R}^{d\times d}$ and scaling vector $\B{\alpha}_k=(\alpha_{k,1}, \cdots, \alpha_{k,N_k})\in \mathbb{R}_+^{N_k}$.
If for every positive integer $n$ and every $k=1, \cdots, m$, the kernel
$\big(\U{K}_n^{\B{X}_k}(\B{t}_1, \cdots, \B{t}_n)\bigr)^{\frac{m-1}{m}}$ is integrable over $\bigl(\mathcal{I}^{N_k}\bigr)^{n}$, then the $m$-fold intersection local time of $\{\B{X}_{k}\}_{k=1}^m$ on the interval $[0,1]$ exists, and
if moreover, every pair of $\alpha_{k,i}$ and $\alpha_{l,j}$ are mutually rational, i.e., $\frac{\alpha_{k,i}}{\alpha_{l,j}}\in \mathbb{Q}$, then denoting the $m$-fold intersection local time of $\{\B{X}_k\}_{k}$ by $\mathfrak{I}_{m}$, the following limits exists
$$
\lim_{y\rightarrow +\infty}\frac{-\log \mathbb{P}(\mathfrak{I}_{m}>y)}{y^{\frac{1}{\gamma}}}
\quad
\text{and}
\quad
\lim_{n\rightarrow +\infty}\frac{\sqrt[n]{\mathbb{E}\bigl((\mathfrak{I}_{m})^n\bigr)}}
{n^{\gamma}},
$$
where $\gamma:=\frac{(m-1)\textrm{tr}(\BU{H})}
{\sum_{k=1}^{m}\sum_{i=1}^{N_k}\alpha_{k,i}}$.
\end{cor}
\begin{proof}
In Section \ref{Existence of the limit}.
\end{proof}

\begin{remark}
Although this theorem affirms the convergence of $\{\frac{\sqrt[n]{\mathbb{E}\bigl((\mathfrak{I}_{m})^n\bigr)}}
{n^{\gamma}}\}_n$, it does not guarantee that $\gamma=\frac{(m-1)\textrm{tr}(\BU{H})}
{\sum_{i=1}^{m}\sum_{k=1}^{N_i}\alpha_{i,k}}$ is the right exponent for the growth of $\sqrt[n]{\mathbb{E}\bigl((\mathfrak{I}_{m})^n\bigr)}$. For that, we would also need the finiteness of the above limit. In Corollary \ref{Exceptional Case}, we saw some examples where the right growth exponent is larger than the exponent given by this theorem ($H$ instead of $H/2$).
\end{remark}



\section{Definition of local times: Pointwise versus functional }\label{Two different definitions of local times}
Let $\B{X}$ be an $(N,d)$-random field over a cube $\mathcal{C}({\B{p}},\B{T})$ in $\mathbb{R}^N$. Let $\mu$ be the occupation measure of $\B{X}$, i.e. for every Borelian subset $\mathcal{B}\subseteq\mathbb{R}^d$ we have
$$
\mu_{\B{X}}(\mathcal{B})=\lambda_N(\{\B{t}\in\mathcal{C}({\B{p}},\B{T});\; \B{X}_{\B{t}}\in\mathcal{B}\}),
$$
where $\lambda_N$ denotes the Lebesgue measure on $\mathbb{R}^N$.
\begin{dfn}[Functional definition of local time]
If the occupation measure of $\B{X}$ is almost-surely absolutely continuous with respect to the Lebegue measure on $\mathbb{R}^d$, i.e. when $\mu_{\B{X}}\ll \lambda_d$, its Radon-Nikodym derivative is called the local time (or occupation density) of $\B{X}$ over $\mathcal{C}({\B{p}},\B{T})$, and we denote it by we denote this function by $\bar{L}_{\B{X}}$.
\end{dfn}
Clearly, this definition does not provide a unique pointwise definition for the local time but a set of functions that are equal to each other almost surely. It is also clear that for any positive (or bounded) measurable function $f:\mathbb{R}^d\rightarrow \mathbb{R}$, we have
$$
\int_{\mathcal{C}(\B{p},\B{T})}f(\B{X}_{\B{t}})\,\mathrm{d}\B{t}=
\int_{\mathbb{R}^d}f(\B{x}) \bar{L}_{\B{X}}(\B{x})\,\mathrm{d}\B{x}).
$$
A sufficient condition for the existence of the local time as defined above, is the following
\begin{equation}\label{sufficient condition for existence of functional local time}
\int_{\mathcal{C}(\B{p},\B{T})}\int_{\mathcal{C}(\B{p},\B{T})} \frac{1}{\sqrt{\detcov(\B{X}_{\B{t}}-\B{X}_{\B{s}})}}\,
\mathrm{d}\B{t}\,\mathrm{d}\B{s}<+\infty;
\end{equation}
see e.g. \cite{Pitt78} or \cite{GemanHorowitz1980}. By Corollary \ref{Reduction Inequality for detCov}, we have
$$
\detcov[\B{X}_{\B{t}} \B{X}_{\B{s}}]\leq \detcov(\B{X}_{\B{t}})  \detcov(\B{X}_{\B{t}}-\B{X}_{\B{s}}).
$$
So it is clear that conditions $\mathfrak{A}_1$ and $\mathfrak{A}_0$ imply Equation \eqref{sufficient condition for existence of functional local time}. So assuming these two conditions, we have the existence of the local times, both in the pointwise and functional definitions. Moreover, for every $\B{x}\in\mathbb{R}^d$ we have
$$
\int_{\mathcal{B}_\varepsilon(\B{x})} \bar{L}_{\B{X}}(\B{y})\,\mathrm{d}\B{y}=
\int_{\mathcal{C}(\B{p},\B{T})}
\mathbf{1}_{\|\B{X}_{\B{t}}-\B{x}\|<\varepsilon}(\B{t}) \,\mathrm{d}\B{t},
$$
hence
$$
\lim_{\varepsilon\rightarrow0}\frac{1}{\vol(\mathcal{B}_\varepsilon(\B{x}))}
\int_{\mathcal{B}_\varepsilon(\B{x})} \bar{L}_{\B{X}}(\B{y})\,\mathrm{d}\B{y}
=L_{\B{x}}(\B{X};\mathcal{C}(\B{p},\B{T})),
$$
which means that irrespective of the chosen version of $\bar{L}_{\B{X}}$, its pointwise local mean equals the local time defined in the pointwise manner.

\section{Formulation in moments growth rate}\label{Formulation in moments growth rate}
The following theorem is a special case of Kasahara's Tauberian theorem \cite[Theorem 4]{Kasahara78} which relates the probability tail behavior to the moments asymptotic behavior.

\begin{thm}[Kasahara 1978]\label{Kasahara}
For any positive random variable $Y$, any positive number $\lambda$, and any $A\in(0,+\infty]$,
the limit
$$
\lim_{n\rightarrow +\infty}\frac{\sqrt[n]{\mathbb{E}(Y^n)}}{n^\lambda}
$$
exists and equals $A$ if and only if the limit
$$
\lim_{x\rightarrow +\infty}\frac{-\log \mathbb{P}(Y>x)}{x^{\frac{1}{\lambda}}}
$$
exists and equals $\frac{\lambda}{e A^{\frac{1}{\lambda}}}$.
\end{thm}
We aim to prove that for any centered Gaussian $(N,d)$-field $\B{X}$ satisfying conditions $\mathfrak{A}_1$, $\mathfrak{A}_0$, $\mathfrak{A}_2$, and $\mathfrak{A}_3$, the following limit exists
$$
\lim_{n\rightarrow +\infty}\frac{\sqrt[n]{\mathbb{E}(Z^n)}}{n^\lambda},
$$
where $Z:=L_{\B{0}}(\B{X},\mathcal{I}^N)$ and $\lambda:=\frac{ \textrm{tr}(\BU{H})}{\sum_{k=1}^{N}\alpha_k}$. Clearly this along with the above theorem proves the existence of the following limit with $\lambda:=\frac{ \textrm{tr}(\BU{H})}{\sum_{k=1}^{N}\alpha_k}$.
$$
\lim_{x\rightarrow +\infty}
x^{-\frac{1}{\lambda}}
\log\mathbb{P}[L_{\B{0}}(\B{X},\mathcal{I}^N)>x].
$$

We have the following proposition on the existence of the local time of zero-mean Gaussian fields at the origin and its moments. In the proof we use some arguments of \cite{Pitt78}.
\begin{prop}\label{existence and approximation}
For any Gaussian field $\B{X}$ satisfying condition $\mathfrak{A}_1$,  the local times $Z_{\B{x}}:=L_{\B{x}}(\B{X},\mathcal{I}^N)$ exist for every $\B{x}\in\mathbb{R}^d$, and we have
$$
\mathbb{E}(Z_{\B{x}}^n)\leq \int_{\prod_{i=1}^n \mathcal{I}^N} \U{K}_n^{\B{X}}(\B{t}_1, \cdots, \B{t}_n) \,\mathrm{d}\B{t}_1\cdots \mathrm{d}\B{t}_n\;:\quad \forall \B{x}\in\mathbb{R}^d,
$$
and
$$
\mathbb{E}(Z_{\B{0}}^n)= \int_{\prod_{i=1}^n \mathcal{I}^N} \U{K}_n^{\B{X}}(\B{t}_1, \cdots, \B{t}_n) \,\mathrm{d}\B{t}_1\cdots \mathrm{d}\B{t}_n,
$$
where $\prod_{i=1}^n\mathcal{I}^N$ denotes the $n$-times Cartesian product $\mathcal{I}^N\times \cdots\times \mathcal{I}^N$, and $\U{K}_n^{\B{X}}(\B{t}_1, \cdots, \B{t}_n)$ is as defined in Equation \eqref{K_n Definition}.
%
\end{prop}
\end{remark}
\begin{proof}
We prove the proposition for $\B{x}=\B{0}$. The proof for the general $\B{x}$ is similar. For the rest of the proof, we denote $Z_{\B{0}}$ simply by $Z$. Let $\|\cdot\|$ be some arbitrary norm on $\mathbb{R}^d$ and define $Z_\varepsilon:=\int_{\mathcal{I}^N}\frac{1}{V_\varepsilon}
\mathbf{1}_{\{\|\B{X}_{\B{t}}\|<\varepsilon\}}(\B{t})\mathrm{d}\B{t}$, where $V_\varepsilon$ denotes the volume of the $d$-dimensional ball $\{x\in\mathbb{R}^d ; \|x\|<\varepsilon\}$, and $\mathbf{1}_{\{\cdot\}}$ denotes the indicator function. First we show that $\{Z_\varepsilon\}$ is cauchy in $\mathrm{L}^n(\Omega, \mathbb{P})$. Indeed, let $\mathfrak{S}$ be the set of all possible functions from $\{1, \cdots, n\}$ into $\{0,1\}$. For any function $\sigma\in \mathfrak{S}$, we define $\xi_i^\sigma$ to be equal to $\varepsilon$ if $\sigma(i)=0$ and be equal to $\delta$ if $\sigma(i)=1$. It is then easy to verify the following equality
$$
\mathbb{E}\bigl[(Z_\varepsilon-Z_\delta)^n\bigr]=
\int_{\prod_{i=1}^kI^N} \sum_{\sigma\in \mathfrak{S}} \frac{(-1)^{\sum_{i=1}^n\sigma(i)}}{\Pi_{i=1}^n V_{\xi_i^\sigma}}
\mathbb{P}\bigl(\bigcap_{i=1}^n \{|\B{X}_{\B{t}_i}\|<\xi_i^\sigma\})
\mathrm{d}\B{t}_1\cdots \mathrm{d}\B{t}_n.
$$
As the variables $\{\B{X}_{\B{t}}\}_{{\B{t}}}$ are jointly normal with mean zero, for every $\sigma\in \mathfrak{S}$ and each fixed $\B{t}_1, \cdots, \B{t}_n$, we have
$$
\frac{1}{\Pi_{i=1}^n V_{\xi_i^\sigma}}
\mathbb{P}\bigl(\|\B{X}_{\B{t}_1}\|<\xi_1^\sigma , \ldots, \|\B{X}_{\B{t}_n}\|<\xi_n^\sigma\bigr)\leq
\U{K}_n^{\B{X}}(\B{t}_1, \cdots, \B{t}_n)
$$
and
$$
\frac{1}{\Pi_{i=1}^n V_{\xi_i^\sigma}}
\mathbb{P}\bigl(\|\B{X}_{\B{t}_1}\|<\xi_1^\sigma , \ldots, \|\B{X}_{\B{t}_n}\|<\xi_n^\sigma\bigr)
\overset{\varepsilon, \delta\downarrow 0}{\longrightarrow}
\U{K}_n^{\B{X}}(\B{t}_1, \cdots, \B{t}_n)
$$
Noting that $\sum_{\sigma\in \mathfrak{S}}(-1)^{\sum_{i=1}^n\sigma(i)}=(1-1)^n=0$, by dominated convergence and Property $\mathfrak{A}_1$ we have
$$
\mathbb{E}\bigl[(Z_\varepsilon-Z_\delta)^n\bigr]\overset{\varepsilon, \delta\downarrow 0}{\longrightarrow} 0,
$$
which proves that $Z_\varepsilon$ is $\mathrm{L}^n(\Omega, \mathbb{P})$-Cauchy whose limit is noting but $Z$. Similarly one can also show that
$$
\mathbb{E}(Z^n)=\lim_{\varepsilon\rightarrow 0}\mathbb{E}(Z_\varepsilon^n)=\int_{\prod_{i=1}^n \mathcal{I}^N} \U{K}_n^{\B{X}}(\B{t}_1, \cdots, \B{t}_n)\, \mathrm{d}\B{t}_1\cdots \mathrm{d}\B{t}_n.
$$
\end{proof}


\section{Upper bound}\label{Section on Upper bounds}
In this section we prove Propositions \ref{Equivalence of self-similar functions}, \ref{parameters restrictions}, Lemma \ref{local nondeterminism theorem}, and finally Proposition \ref{Upper bound for intersection local times}.

\begin{proof}[Proof of Proposition \ref{Equivalence of self-similar functions}]
First we normalize the vector $\B{\alpha}=(\alpha_1, \alpha_2, \cdots, \alpha_N)$ and $H$ so that for every $i$ we have $\alpha_i\geq 1$. This is possible because $f(\B{x}\circ \omega^{\B{\alpha}/p})=\omega^{H/p} f(\B{x})$ for every $p>0$ as well. This turns the $\B{\alpha}$-distance $d(\B{x},\B{y}):=\|\B{x}-\B{y}\|_{\B{\alpha}}$ into a translation-invariant metric on $\mathbb{R}^N$.\\
i) We denote the $\ell^2$-norm on $\mathbb{R}^N$ by $\|\cdot\|_2$, and take the standard definition for the $\B{\alpha}$ and $\ell^2$-balls centered at the origin of radius $r>0$ as $B_{\B{\alpha}}(\B{0},r):=\{\B{x}\in\mathbb{R}^N :\,\|\B{x}\|_{\B{\alpha}}<r\}$ and $B_{2}(\B{0},r):=\{\B{x}\in\mathbb{R}^N :\,\|\B{x}\|_{2}<r\}$. It is easy to verify that every $\B{\alpha}$-ball centered at the origin contains a $\ell^2$-ball centered at the origin, and vice versa. This shows that the ${\B{\alpha}}$-metric induces the same topology on $\mathbb{R}^N$ as the standard topology induced by the $\ell^2$-norm. In particular it means that the function $f$ is also continuous under the ${\B{\alpha}}$-metric. So there exists $\varepsilon>0$ such that for every $\B{x}$ we have $|f(\B{x})|<1$ if $\|\B{x}\|_{\B{\alpha}}<\varepsilon$. Now for any $\B{x}\in \mathbb{R}^N$, choose $\omega>0$ such that $\|\B{x}\circ \omega^{\B{\alpha}}\|_{\B{\alpha}}=\omega \|\B{x}\|_{\B{\alpha}}=\varepsilon/2$. It follows that $|f(\B{x}\circ \omega^{\B{\alpha}})|<1$ which means $|f(\B{x})|<\frac{2^H}{\varepsilon^H}\|\B{x}\|_{\B{\alpha}}^H$.\\
ii) As the set $\mathcal{S}_{\B{\alpha}}:=\{\B{x}\in \mathbb{R}^N: \; \|\B{x}\|_{\B{\alpha}}=1\}$ is compact, the image of the continuous function $f$ achieves a minimum over $\mathcal{S}_{\B{\alpha}}$ which is strictly positive. In other words, there exists $\delta>0$ such that for every $\B{x}\in \mathcal{S}_{\B{\alpha}}$ we have $f(\B{x})\geq \delta$. Now for a general $\B{x}\in \mathbb{R}^N$, we can choose $\omega>0$ such that $\|\B{x}\circ \omega^{\B{\alpha}}\|_{\B{\alpha}}=1$. So we have $|f(\B{x}\circ \omega^{\B{\alpha}})|\geq \delta$, hence $|f(\B{x})|\geq \delta \|\B{x}\|_{\B{\alpha}}^H$.
\end{proof}

\begin{proof}[Proof of Proposition \ref{parameters restrictions}]
We prove it for the case of $n=1$. The proof for larger $n$'s is similar.\\
Let the surface $\mathcal{S}_{\B{\alpha}}^+\subset\mathbb{R}^N$ be defined as follows
$$
\mathcal{S}_{\B{\alpha}}^+:=\{(x_1, x_2, \cdots, x_N)\in \mathbb{R}_{+}^N: \; \sum_{i=1}^{N} \sqrt[\alpha_i]{x_i}=1\}.
$$
By diagonal self-similarity (Property $\mathfrak{A}_3$) and noting Remark \ref{diagonal self-similarity covariance formulation}, for any
$\B{t}\in \mathbb{R}^N$ and any $\omega>0$ we have
$$
\detcov (\B{X}_{\omega^{\B{\alpha}}\circ\B{t}})=\det (\omega^{\BU{H}})\,
\detcov(\B{X}_{\B{t}})\,
\det(\omega^{\BU{H}^\dagger})=
\omega^{2\textrm{tr}(\BU{H})}\,\detcov(\B{X}_{\B{t}}), $$
where we used the fact that the determinant of the exponential of a matrix equals the exponential of its trace, i.e., $\det(e^{\BU{H}})=e^{\textrm{tr}(\BU{H})}$ (see e.g. \cite[ch.2]{Hall2003}). Suppose $\B{s}\mapsto\B{\sigma}_{\B{s}}$ is an arbitrary parametrization of the surface $\mathcal{S}_{\B{\alpha}}^+$, where $\B{\sigma}_{\B{s}}=(\sigma_1(\B{s}), \cdots,\sigma_N(\B{s}))$ and $\B{s}=(s_1, \cdots, s_{N-1})$. Then using the change of variables $(\omega,\B{s})$ with $\B{t}=\omega^{\B{\alpha}}\circ\B{\sigma}_{\B{s}}$, we have
$$
\begin{aligned}
\int_{[0,1]^N} \bigl(\U{K}_1(\B{t})\bigr)^{\beta}\, \mathrm{d}\B{t}&\geq
\int_{0}^1 \int_{\B{\sigma}^{-1}(\mathcal{S}_{\B{\alpha}}^+)} \omega^{\sum_{i=1}^N\alpha_i-1}
\bigl(\U{K}_1(\omega^{\B{\alpha}}\circ\B{\sigma}_{\B{s}})\bigr)^{\beta}\, J_\sigma(\B{s}) \mathrm{d}\B{s}\,\mathrm{d}\omega\\
&=\int_{0}^1 \frac{\omega^{\sum_{i=1}^N\alpha_i-1}}{\omega^{\beta\, \textrm{tr}(\BU{H})}}\mathrm{d}\omega
\int_{\B{\sigma}^{-1}(\mathcal{S}_{\B{\alpha}}^+)} \bigl(\U{K}_1(\B{\sigma}_{\B{s}})\bigr)^{\beta}\, J_\sigma(\B{s})\mathrm{d}\B{s},
\end{aligned}
$$
where $J_\sigma(\B{s})$ is the absolute value of the following determinant
\begin{equation}\label{J_sigma}
\begin{vmatrix}
  \alpha_1 \sigma_1 & \frac{\partial \sigma_1}{\partial s_1} & \frac{\partial \sigma_1}{\partial s_2} & \dots & \frac{\partial \sigma_1}{\partial s_{N-1}} \\
  \vdots & \vdots & \vdots& \dots & \vdots \\
  \alpha_N \sigma_N & \frac{\partial \sigma_N}{\partial s_1} & \frac{\partial \sigma_N}{\partial s_2}& \dots & \frac{\partial \sigma_N}{\partial s_{N-1}}
\end{vmatrix}.
\end{equation}
%
This implies that if $\bigl(\U{K}_1(\B{t})\bigr)^{\beta}$ is integrable over the cube $[0,1]^N$, then the following inequality has to hold true
$$
\sum_{i=1}^N\alpha_i> \beta\, \textrm{tr}(\BU{H}).
$$
\end{proof}

Now we turn to the proof of Lemma \ref{local nondeterminism theorem}. First we need a definition
\begin{dfn}\label{narrowing definition}
Consider $n\in\mathbb{N}$, and a sequence of vectors $\B{t}_1$, ..., $\B{t}_n\in\mathbb{R}^N$, and let $d$ be a metric on $\mathbb{R}^N$. We say that the sequence is `narrowing' with respect to the metric $d$, if for any $k=1, \cdots, n-1$ we have
$$
d(\B{t}_{k+1}, \B{t}_k)=\min_{i=1,\cdots, k}d(\B{t}_{k+1},\B{t}_{i}).
$$
\end{dfn}
\begin{remark}\label{arranging into narrowing}
For any fixed metric $d$, every finite subset $\mathcal{F}\subset\mathbb{R}^N$ of $n$ points can be arranged in such a way that the ordered sequence is narrowing with respect to the metric $d$. Indeed, the procedure is simple: start with an arbitrary point in $\mathcal{F}$ and label it $\B{t}_{n}$. Having chosen $\B{t}_{n}$, $\B{t}_{n-1}$, ..., $\B{t}_{k}$, choose among the remaining  points, i.e., from $\mathcal{F}\setminus\{\B{t}_{n}, \B{t}_{n-1}, ...,\B{t}_{k}\}$, the one which is the closest to $\B{t}_{k}$ and label it $\B{t}_{k-1}$.
\end{remark}

We will need a theorem from \cite{Tassiulas1997} which relates our problem to the nearest neighbor strategy for solving the travelling salesman problem. We proceed with some preliminary definitions from \cite{Tassiulas1997}.

For a given metric $d$ on $\mathbb{R}^N$, the diameter of a subset $\mathcal{C}\subseteq\mathbb{R}^N$ is defined as
$$
D_d(\mathcal{C}):=\sup_{\B{x}, \B{y}\in \mathcal{C}} d(\B{x}, \B{y}).
$$
A family $\mathcal{P}=\{\mathcal{C}_l\}_{l=1}^{P}$ of subsets of $\mathbb{R}^N$ is called a covering of subset $\mathcal{A}\subset\mathbb{R}^N$ if $\mathcal{A}\subseteq\bigcup_{l=1}^{P}\mathcal{C}_l$. The diameter of a covering $\mathcal{P}$, denoted by $D_d(\mathcal{P})$ is defined as the maximum diameter of its elements, i.e.
$$
D_d(\mathcal{P})=\max_{l=1,\cdots,P} D_d(\mathcal{C}_l).
$$
For any $n$ distinct points in $\mathbb{R}^N$, an arrangement (ordering) $\B{t}_{1}$, $\B{t}_{2}$, ..., $\B{t}_{n}$ of the points is a `nearest neighbor tour' if it satisfies the following property
$$
d(\B{t}_{k+1}, \B{t}_{k})=\min_{i=k+1, \cdots, n} d(\B{t}_{k}, \B{t}_{i})\;: \quad \forall k=1,\cdots, n-1.
$$
For every nearest neighbor tour $\mathfrak{T}$ of $n$ points, one defines its (loop) length as
$$
\mathbb{L}_d(\mathfrak{T}):=\sum_{k=1}^{n}d(\B{t}_{k+1}, \B{t}_{k}),
$$
where $\B{t}_{n+1}:=\B{0}$.
We are interested in finding a general upper bound on this length when all the points of the tour are required to lie inside the unit cube $\mathcal{I}^N$. Indeed, the `worst case length' of $n$-point nearest neighbor (NN) tours over a subset $\mathcal{A}\subset\mathbb{R}^N$ is defined as follows
$$
\mathbb{L}_d(n;\mathcal{A}):=\sup_{\substack{\mathcal{F}\subset\mathcal{A}\\|\mathcal{F}|=n}}
\max_{\mathfrak{T}\in \textrm{NN}(\mathcal{F})} \mathbb{L}_d(\mathfrak{T}),
$$
where $|\mathcal{F}|$ denotes the cardinality of the set $\mathcal{F}$, and $\textrm{NN}(\mathcal{F})$ is the set of all possible nearest neighbor tours of the points of $\mathcal{F}$.
Then we have the following theorem from \cite{Tassiulas1997}.
\begin{thm}\label{nearest neighbor covering bound}
Let $d$ be a metric on $\mathbb{R}^N$, and $\mathcal{A}\subset\mathbb{R}^N$. Then for any sequence of coverings $\{\mathcal{P}_m\}_{m=1}^M$ of $\mathcal{A}$ with decreasing diameters, i.e. $D_d(\mathcal{P}_m)\geq D_d(\mathcal{P}_{m+1})$ for every $m=1,\cdots, M-1$. Then the worst case length of $n$-point nearest neighbor tours is bounded as follows
$$
\mathbb{L}_d(n;\mathcal{A})\leq nD_d(\mathcal{P}_M)+\sum_{m=2}^M |\mathcal{P}_m|\bigl(D_d(\mathcal{P}_{m-1})-D_d(\mathcal{P}_m)\bigr)+
|\mathcal{P}_1|\bigl(D_d(\mathcal{A})-D_d(\mathcal{P}_1)\bigr),
$$
where $|\mathcal{P}_m|$ denotes the cardinality of $\mathcal{P}_m$.
\end{thm}
We should mention that this theorem has been stated in \cite{Tassiulas1997} only for the case where the metric $d$ comes from a norm on $\mathbb{R}^N$. Nevertheless, with a careful examination of their proof one can verify that their arguments work even when $d$ is a general metric on $\mathbb{R}^N$.

Now we are ready to prove the following theorem.
\begin{thm}\label{nearest neigbor bound theorem}
Let $N\in\mathbb{N}^{\geq 2}$ and $\B{\alpha}=(\alpha_1, \alpha_2, \cdots, \alpha_N)\in [1,\infty)^N$. Then the worst case length of nearest neighbor tours of $n$-point subsets of $[0,1]^N$ with respect to the metric $\|\cdot\|_{\B{\alpha}}$ defined in
\eqref{alpha distance}, is bounded as follows
$$
\mathbb{L}_{\B{\alpha}}(n,[0,1]^N)\leq c_{\B{\alpha}}\, n^{1-\frac{1}{\sum_{i=1}^{N}\alpha_i}},
$$
where $c_{\B{\alpha}}$ is a constant that depends only on $\B{\alpha}$.
\end{thm}
\begin{proof}
For every $m\in\mathbb{N}$, we define the set of points $\mathcal{G}_m\subset [0,1]^N$ as
$$
\mathcal{G}_m:=\Bigl\{\,(\frac{\ell_1}{m^{\alpha_1}}, \frac{\ell_2}{m^{\alpha_2}}, \cdots, \frac{\ell_N}{m^{\alpha_N}})\,\Big|\, \ell_i\in \{0, 1, \cdots, \lceil m^{\alpha_i}\rceil-1\}:\, \forall i=1, 2, \cdots, N\Bigr\},
$$
where $\lceil\cdot\rceil$ denotes the ceiling function, i.e. $\lceil x\rceil$ is the smallest integer that is larger than or equal to $x$.
For every $\B{p}\in\mathcal{G}_m$, we define the sub-cube $\mathcal{C}(\B{p},m^{-\B{\alpha}}) $ as usual, i.e.
$$
\mathcal{C}(\B{p},m^{-\B{\alpha}}) :=
\B{p}+\prod_{i=1}^{N}[0,\frac{1}{m^{\alpha_i}}].
$$
So for any positive integer $m$ we define the following covering of the set $[0,1]^N$
$$
\mathcal{P}_m:=\{\mathcal{C}(\B{p},m^{-\B{\alpha}})\;: \quad\B{p}\in \mathcal{G}_m\}.
$$
We note that for every $m$ we have
$$
D(\mathcal{P}_m)=\frac{N}{m},
\qquad \text{and} \qquad
|\mathcal{P}_m|=\prod_{i=1}^{N}\lceil m^{\alpha_i}\rceil.
$$
So by Theorem \ref{nearest neighbor covering bound}, for every $M\in\mathbb{N}$ we have
$$
\begin{aligned}
\mathbb{L}_{\B{\alpha}}(n,[0,1]^N)
&
\leq \frac{n N}{M}+N \sum_{m=2}^{M}(\frac{1}{m-1}-\frac{1}{m})\prod_{i=1}^{N}\lceil m^{\alpha_i}\rceil
\\&
\leq
\frac{n N}{M}+N 2^{N+1}\sum_{m=2}^{M} m^{\sum_{i=1}^N\alpha_i-2}
\end{aligned}
$$
Due to the assumption, we have $\sum_{i=1}^N\alpha_i\geq 2$, and hence
$$
\begin{aligned}
\mathbb{L}_{\B{\alpha}}(n,[0,1]^N)
\leq
\frac{n N}{M}+N 2^{N+1}\int_{1}^{M} m^{\sum_{i=1}^N\alpha_i-2}
\leq
\frac{n N}{M}+\frac{N 2^{N+1}}{\sum_{i=1}^N\alpha_i-1} M^{\sum_{i=1}^N\alpha_i-1}.
\end{aligned}
$$
If we choose $M:=n^{\frac{1}{\sum_{i=1}^N\alpha_i}}$ we get the desired bound.
\end{proof}

So we are ready to prove Lemma \ref{local nondeterminism theorem}.

\begin{proof}[Proof of Lemma \ref{local nondeterminism theorem}]
Choose $\B{\alpha}=(\alpha_1, \alpha_2, \cdots, \alpha_N)\in\mathbb{R}^N_+$ and $H>0$ such that for every $i=1, \cdots, N$ we have $\alpha_i\geq 1$ and $\alpha_i=H\xi_i$. First we prove the theorem for the case of $N\geq 2$.
We define $\mathcal{N}_{n, \B{\alpha}}$ as the set of all $nN$-tuples $(\B{t}_{1}, \B{t}_{2}, \cdots, \B{t}_{n})\in \mathcal{I}^N\times\cdots\times\mathcal{I}^N$ such that the sequence $\B{t}_{1}$, $\B{t}_{2}$, ..., $\B{t}_{n}$ is narrowing with respect to the $\B{\alpha}$-distance defined by $\|\cdot\|_{\B{\alpha}}$ in Equation \eqref{alpha distance}. Then by Remark \ref{arranging into narrowing}, and the fact that
$\U{K}_n^{\B{X}}(\B{t}_1,\B{t}_2, \cdots, \B{t}_n)$ is symmetric with respect to its arguments, i.e. it is permutation-invariant, we have
\begin{equation}\label{permutation sum}
\int_{(\mathcal{I}^N)^n}\big(\U{K}_n^{\B{X}}(\B{t}_1, \cdots, \B{t}_n)\bigr)^{\beta}\mathrm{d}\B{t}_1\cdots\mathrm{d}\B{t}_n\leq
n! \int_{\mathcal{N}_{n, \B{\alpha}}}\big(\U{K}_n^{\B{X}}(\B{t}_1, \cdots, \B{t}_n)\bigr)^{\beta}\mathrm{d}\B{t}_1\cdots\mathrm{d}\B{t}_n.
\end{equation}
By strong local nondeterminism, for any narrowing sequence $\B{t}_1$,$\B{t}_2$, ..., $\B{t}_n$ we have
$$
\detcov[\B{X}_{\B{t}_k}|\B{X}_{\B{t}_1}, \B{X}_{\B{t}_2},\cdots, \B{X}_{\B{t}_{k-1}}]\geq C \min\{
\|\B{t}_k-\B{t}_{k-1}\|_{\B{\alpha}}^{2H},
\|\B{t}_k\|_{\B{\alpha}}^{2H}\}\;:\quad \forall k=2, \cdots,n.
$$
So by Proposition \ref{Inequality for detCov} we have
\begin{equation}\label{bounding the kernel by strong nondeterminism}
\U{K}_n^{\B{X}}(\B{t}_1, \cdots, \B{t}_n)\leq c_1^{n}\prod_{k=1}^{n}
\frac{1}{\min\{
\|\B{t}_k-\B{t}_{k-1}\|_{\B{\alpha}}^H,
\|\B{t}_k\|_{\B{\alpha}}^H\}}
\end{equation}
where $c_1:=\frac{1}{\sqrt{C(2\pi)^d}}$.

On the other hand, it is easy to verify that if a sequence $\B{t}_{1}$, $\B{t}_{2}$, ..., $\B{t}_{n}$ is narrowing, then its reversal, i.e. $\B{t}_{n}$, $\B{t}_{n-1}$, ..., $\B{t}_{1}$ satisfies the nearest neighbor property. Hence by Theorem \ref{nearest neigbor bound theorem}, for any $(\B{t}_{1}, \B{t}_{2}, \cdots, \B{t}_{n})\in \mathcal{N}_{n, \B{\alpha}}$, we have
$$
\sum_{k=2}^{n}\|\B{t}_k-\B{t}_{k-1}\|_{\B{\alpha}}\leq \mathbb{L}_{\B{\alpha}}(n; \mathcal{I}^N)\leq c_{\B{\alpha}} n^{1-\frac{1}{\sum_{i=1}^{N}\alpha_i}},
$$
so
\begin{equation}\label{linear in n upper bound on the exponent}
\sum_{k=1}^{n}k^{\frac{1}{\sum_{i=1}^{N}\alpha_i}}\|\B{t}_k-\B{t}_{k-1}\|_{\B{\alpha}}\leq c_2 \, n,
\end{equation}
where $c_2$ is a constant that only depends on $N$, $\B{\alpha}$ and $d$. So \eqref{linear in n upper bound on the exponent} and \eqref{bounding the kernel by strong nondeterminism} imply
$$
\begin{aligned}
&\int_{\mathcal{N}_{n, \B{\alpha}}}\big(\U{K}_n^{\B{X}}(\B{t}_1, \cdots, \B{t}_n)\bigr)^{\beta}\mathrm{d}\B{t}_1\cdots\mathrm{d}\B{t}_n\\
&\quad
\leq
e^{c_2\,n}
\int_{\mathcal{N}_{n, \B{\alpha}}}
c_1^{n\beta}\prod_{k=1}^{n}
\frac{\exp\bigl(-k^{\frac{1}{\sum_{i=1}^{N}\alpha_i}}\min\{
\|\B{t}_k-\B{t}_{k-1}\|_{\B{\alpha}},
\|\B{t}_k\|_{\B{\alpha}}\}\bigr)}
{\min\{
\|\B{t}_k-\B{t}_{k-1}\|_{\B{\alpha}}^{\beta H},
\|\B{t}_k\|_{\B{\alpha}}^{\beta H}\}}\mathrm{d}\B{t}_1\cdots\mathrm{d}\B{t}_n\\
&\quad
\leq
c_3^n
\int_{\mathbb{R}^{nN}}
\prod_{k=1}^{n}
\frac{\exp\bigl(-k^{\frac{1}{\sum_{i=1}^{N}\alpha_i}}\min\{
\|\B{t}_k-\B{t}_{k-1}\|_{\B{\alpha}},
\|\B{t}_k\|_{\B{\alpha}}\}\bigr)}
{\min\{
\|\B{t}_k-\B{t}_{k-1}\|_{\B{\alpha}}^{\beta H},
\|\B{t}_k\|_{\B{\alpha}}^{\beta H}\}}
\mathrm{d}\B{t}_1 \cdots \mathrm{d}\B{t}_n\\
&\quad
\leq
c_3^n
\int_{\mathbb{R}^{nN}}
\prod_{k=1}^{n}
\biggl(
\frac{\exp\bigl(-k^{\frac{1}{\sum_{i=1}^{N}\alpha_i}}
\|\B{t}_k-\B{t}_{k-1}\|_{\B{\alpha}}\bigr)}
{\|\B{t}_k-\B{t}_{k-1}\|_{\B{\alpha}}^{\beta H}}+
\frac{\exp\bigl(-k^{\frac{1}{\sum_{i=1}^{N}\alpha_i}}
\|\B{t}_k\|_{\B{\alpha}}\bigr)}
{\|\B{t}_k\|_{\B{\alpha}}^{\beta H}}\biggr)
\mathrm{d}\B{t}_1 \cdots \mathrm{d}\B{t}_n
\end{aligned}
$$
where $c_3:=e^{c_2}c_1^\beta$.
Let $\mathfrak{S}$ be the set of all possible functions from $\{1, \cdots, n\}$ into $\{0,1\}$, and for any function $\vartheta\in \mathfrak{S}$ and any $k\in\{1,\cdots, n\}$, define
$$
\B{y}_k^\vartheta=
\left\{
\begin{array}{lr}
\B{t}_k, & \text{if } \vartheta(k)=0\\
\B{t}_k-\B{t}_{k-1}, & \text{if } \vartheta(k)=1.
\end{array}\right.
$$
Then the last inequality can be written as follows
\begin{equation}\label{upperbound on integral of kernel}
\begin{aligned}
&\int_{\mathcal{N}_{n, \B{\alpha}}}\big(\U{K}_n^{\B{X}}(\B{t}_1, \cdots, \B{t}_n)\bigr)^{\beta}\mathrm{d}\B{t}_1\cdots\mathrm{d}\B{t}_n\\
&\quad \leq
c_3^n
\sum_{\vartheta\in \mathfrak{S}}
\int_{\mathbb{R}^{nN}}
\prod_{k=1}^{n}
\frac{\exp\bigl(-k^{\frac{1}{\sum_{i=1}^{N}\alpha_i}}
\|\B{y}_k^\vartheta\|_{\B{\alpha}}\bigr)}
{\|\B{y}_k^\vartheta\|_{\B{\alpha}}^{\beta H}}
\mathrm{d}\B{t}_1 \cdots \mathrm{d}\B{t}_n\\
&\quad =
c_3^n
\sum_{\vartheta\in \mathfrak{S}}
\int_{\mathbb{R}^{nN}}
\prod_{k=1}^{n}
\frac{\exp\bigl(-k^{\frac{1}{\sum_{i=1}^{N}\alpha_i}}
\|\B{y}_k^\vartheta\|_{\B{\alpha}}\bigr)}
{\|\B{y}_k^\vartheta\|_{\B{\alpha}}^{\beta H}}
\mathrm{d}\B{y}_1^\vartheta \cdots \mathrm{d}\B{y}_n^\vartheta\\
&\quad =
c_3^n
\sum_{\vartheta\in \mathfrak{S}}
\prod_{k=1}^{n}
\int_{\mathbb{R}^{N}}
\frac{\exp\bigl(-k^{\frac{1}{\sum_{i=1}^{N}\alpha_i}}
\|\B{y}\|_{\B{\alpha}}\bigr)}
{\|\B{y}\|_{\B{\alpha}}^{\beta H}}
\mathrm{d}\B{y}.
\end{aligned}
\end{equation}
Define the surface $\mathcal{S}_{\B{\alpha}}\subset\mathbb{R}^N$ as follows
$$
\mathcal{S}_{\B{\alpha}}:=\{(x_1, x_2, \cdots, x_N)\in \mathbb{R}^N: \; \sum_{i=1}^{N} \sqrt[\alpha_i]{|x_i|}=1\},
$$
and take an arbitrary parametrization $\B{s}\mapsto\B{\sigma}_{\B{s}}$   of $\mathcal{S}_{\B{\alpha}}$, where
$\B{\sigma}_{\B{s}}=(\sigma_1(\B{s}), \cdots,\sigma_N(\B{s}))$ and $\B{s}=(s_1, \cdots, s_{N-1})$. Then using the change of variables $(\omega,\B{s})$ with $\B{y}=r^{\B{\alpha}}\circ\B{\sigma}_{\B{s}}$, we have
\begin{equation}\label{passing to radial and factorizing k out}
\begin{aligned}
\int_{\mathbb{R}^{N}}
\frac{\exp\bigl(-k^{\frac{1}{\sum_{i=1}^{N}\alpha_i}}
\|\B{y}\|_{\B{\alpha}}\bigr)}
{\|\B{y}\|_{\B{\alpha}}^{\beta H}}
\mathrm{d}\B{y}
&=
|\mathcal{S}_{\B{\alpha}}^\circ \, |
\int_{0}^{+\infty}
\frac{\exp\bigl(-r k^{\frac{1}{\sum_{i=1}^{N}\alpha_i}}
\bigr) r^{\sum_{i=1}^{N}\alpha_i}}
{r^{\beta H+1}}
\mathrm{d}r\\
&=|\mathcal{S}_{\B{\alpha}}^\circ|
\frac{k^{\frac{\beta H}{\sum_{i=1}^{N}\alpha_i}}}{k}
\int_{0}^{+\infty}
\frac{e^{-r}}{r}
r^{\sum_{i=1}^{N}\alpha_i-\beta H}
\mathrm{d}r
\end{aligned}
\end{equation}
where
$$
|\mathcal{S}_{\B{\alpha}}^\circ|:=
\int_{\B{\sigma}^{-1}(\mathcal{S}_{\B{\alpha}})} J_\sigma(\B{s})\mathrm{d}\B{s},
$$
and $J_\sigma$ is the absolute value of the determinant given in Equation \eqref{J_sigma}.
%
It is important to note that the right-hand side of \eqref{passing to radial and factorizing k out} is finite only if $\sum_{i=1}^{N}\alpha_i>\beta H$. So by Equations \eqref{upperbound on integral of kernel} and \eqref{passing to radial and factorizing k out} we get
\begin{equation}\label{final upperbound on restricted integral of kernel}
\int_{\mathcal{N}_{n, \B{\alpha}}}\big(\U{K}_n^{\B{X}}(\B{t}_1, \cdots, \B{t}_n)\bigr)^{\beta}\mathrm{d}\B{t}_1\cdots\mathrm{d}\B{t}_n
\leq
\,
c_4^n \, \frac{(n!)^{\frac{\beta H}{\sum_{i=1}^{N}\alpha_i}}}{n!},
\end{equation}
where
$$
c_4:=2 c_3
|\mathcal{S}_{\B{\alpha}}^\circ|
\int_{0}^{+\infty}
\frac{e^{-r}}{r}
r^{\sum_{i=1}^{N}\alpha_i-\beta H}.
$$
So finally, plugging Equation \eqref{final upperbound on restricted integral of kernel} into Equation \eqref{permutation sum} we get
$$
\int_{(\mathcal{I}^N)^n}\big(\U{K}_n^{\B{X}}(\B{t}_1, \cdots, \B{t}_n)\bigr)^{\beta}\mathrm{d}\B{t}_1\cdots\mathrm{d}\B{t}_n\leq
c_4^n \, (n!)^{\frac{\beta H}{\sum_{i=1}^{N}\alpha_i}}.
$$
The proof for the case of $N=1$ is similar to the above proof, except for the fact that instead of arranging $\B{t}_i$'s based on the narrowing property, we order them by the natural ordering on $\mathbb{R}$. The rest of the proof remains basically the same.
\end{proof}


\begin{proof}[Proof of Theorem \ref{Upper bound for intersection local times}]
Let us define $\tilde{d}:=(m-1)d$, $\tilde{N}:=\sum_{i=1}^{m}N_i$, and the following vector
$$
\B{\Delta}_{\B{\tilde{t}}}:=\bigl(\B{X}_1(\B{t_1})-\B{X}_2(\B{t_2}), \B{X}_2(\B{t_2})-\B{X}_3(\B{t_3}),\cdots, \B{X}_{m-1}(\B{t_{m-1}})-\B{X}_m(\B{t_m})\bigr)\in\mathbb{R}^{\tilde{d}}.
$$
Evidently, $\B{\Delta}_{\B{\tilde{t}}}$ is a centered Gaussian $(\tilde{N},\tilde{d})$-field.

Take an arbitrary positive integer $n$, and consider any family of points $\{\B{t}_k^i\}_{i,k}$ where $k=1,\cdots,m$, $i=1,\cdots, n$, such that for every $i$ and $k$ we have $\B{t}_k^i\in\mathcal{I}^{N_k}$. Note that the superscript $i$ in $\B{t}_k^i$ is simply an index and should not be confused with exponent. For any such family of points, and for any $i=1,\cdots, n$, we define the vector $\B{\tilde{t}}_i:=(\B{t}_1^i, \cdots, \B{t}_m^i)\in \mathbb{R}^{\tilde{N}}$.
So we are interested in
$$
\detcov[\B{\Delta}_{\B{\tilde{t}}_1}, \B{\Delta}_{\B{\tilde{t}}_2}, \cdots, \B{\Delta}_{\B{\tilde{t}}_n}],
$$
where $\B{\tilde{t}}_1, \cdots, \B{\tilde{t}}_n\in \mathcal{I}^{\tilde{N}}$.

We first note that the determinant of the covariance matrix of a random vector is invariant under permutations of the entries of the random vector. In other words, for any $d$ dimensional random vector $(Y_1,\cdots,Y_d)$, and $\sigma$ an arbitrary permutation of the set $\{1,2,\cdots,d\}$, we have
$$
\detcov[Y_{\sigma(1)}, Y_{\sigma(2)}, \cdots, Y_{\sigma(d)}]=
\detcov[Y_{1}, Y_{2}, \cdots, Y_{d}].
$$
This is true because any permutation of a vector is equivalent to multiplying it with a permutation matrix, and the determinant of any permutation matrix is $\pm1$.
Using this fact, it is easy to verify the following equality
$$
\detcov[\B{\Delta}_{\B{\tilde{t}}_1}, \B{\Delta}_{\B{\tilde{t}}_2}, \cdots, \B{\Delta}_{\B{\tilde{t}}_n}]=\detcov[\B{A}_1-\B{A}_2, \B{A}_2-\B{A}_3, \cdots, \B{A}_{m-1}-\B{A}_m],
$$
where the random vectors $\B{A}_1$, ..., $\B{A}_m$ are defined as follows
$$
\B{A}_k:=
\bigl(\B{X}_k(\B{t}_k^1),\B{X}_k(\B{t}_k^2),\cdots,\B{X}_k(\B{t}_k^n)\bigl)\in \mathbb{R}^{nd}\;;
\quad \forall k=1,\cdots,m.
$$

Note that for any positive integer $m$, any sequence of Gaussian random vectors of the same size $\B{A}_1, \B{A}_2, \cdots,\B{A}_m$, and for any $k=1,\cdots, m$, we have
$$
\detcov\bigl[\B{A}_1, \cdots, \B{A}_m]=\detcov[\B{A}_1, \cdots, \B{A}_{k-1}, \B{A}_k+\sum_{j\neq k} c_j\B{A}_j, \B{A}_{k+1}, \cdots, \B{A}_m\bigr]
$$
where $\sum_{j\neq k}c_j\B{A}_j$ is any arbitrary linear combination of all the involved vectors excluding $\B{A}_k$. Using this simple fact we have
$$
\begin{aligned}
&\detcov\bigl[\B{A}_1-\B{A}_k, \cdots, \B{A}_{k-1}-\B{A}_k, \B{A}_{k+1}-\B{A}_k, \cdots, \B{A}_{m}-\B{A}_k\bigr]=\\
&\qquad=\detcov\bigl[\B{A}_1-\B{A}_2, \B{A}_2-\B{A}_3, \cdots, \B{A}_{m-1}-\B{A}_m\bigr]\quad:\quad \forall k=1, \cdots, m.
\end{aligned}
$$
So using Proposition \ref{Reduction Inequality for detCov}, for every $k$ we have
\begin{equation}\label{simple lower bound for intersection detCov}
\detcov[\B{A}_1-\B{A}_2, \B{A}_2-\B{A}_3, \cdots, \B{A}_{m-1}-\B{A}_m]\geq
\frac{\detcov[\B{A}_1, \B{A}_2, \cdots, \B{A}_m]}{\detcov[\B{A}_k]}.
\end{equation}
Let $p_k:=1-q_k$ for every $k=1,\cdots, m$.
By \eqref{simple lower bound for intersection detCov}, and noting that $\sum_{k=1}^{m}p_k=1$, we have
$$
\detcov[\B{A}_1-\B{A}_2, \B{A}_2-\B{A}_3, \cdots, \B{A}_{m-1}-\B{A}_m]\geq
\prod_{k=1}^{m}\Bigl(\frac{\detcov[\B{A}_1, \B{A}_2, \cdots, \B{A}_m]}{\detcov[\B{A}_k]}\Bigr)^{p_k}.
$$
Using the independence of $\B{A}_i$'s, we get
$$
\detcov\bigl[\B{A}_1-\B{A}_2, \B{A}_2-\B{A}_3, \cdots, \B{A}_{m-1}-\B{A}_m\bigr]\geq
\prod_{k=1}^{m}\bigl(\detcov[\B{A}_k]\bigr)^{\sum_{j\neq k}p_j}.
$$
As $\sum_{j\neq k}p_j=q_k$, we come to
\begin{equation}\label{upper bound on intersection kernel}
\U{K}_n^{\B{\Delta}}(\B{\tilde{t}}_1, \cdots, \B{\tilde{t}}_n)\leq
\prod_{i=1}^{m}\bigl(\U{K}_n^{\B{X}_k}(\B{t}_k^1, \cdots, \B{t}_k^n)\bigr)^{q_k}
\end{equation}
Applying Lemma \ref{local nondeterminism theorem}, it is clear that $\U{K}_n^{\B{\Delta}}(\B{\tilde{t}}_1, \cdots, \B{\tilde{t}}_n)$ is integrable over $\mathcal{I}^{n\tilde{N}}$ and its integral is bounded from above by  $c^n \,(n!)^{\sum_{k=1}^{m} \frac{q_k}{\tilde{\xi}_k}}$, where $c$ is a positive constant that does not depend on $n$.


\end{proof}


\section{Existence of the limit}\label{Existence of the limit}
In this section we will prove that for any centered Gaussian $(N,d)$-field $\B{X}_{\B{t}}$ satisfying conditions $\mathfrak{A}_0$, $\mathfrak{A}_1$, $\mathfrak{A}_2$, and $\mathfrak{A}_3$ (with mutually rational $\alpha_i$'s), the following limit exists
$$
\lim_{n\rightarrow +\infty}\frac{\sqrt[n]{\mathbb{E}(Z^n)}}{n^\lambda},
$$
where $Z:=L_{\B{0}}(\B{X},\mathcal{I}^N)$ and $\lambda:=\frac{ \textrm{tr}(\BU{H})}{\sum_{k=1}^{N}\alpha_k}$, as in the last section. In the sequel, we prove Theorem \ref{main theorem on local times} and Corollary \ref{limit theorem on intersection local times}.

We use the standard notation $\lfloor\cdot\rfloor$ for the floor function ,i.e., for every $x\in \mathbb{R}$, the expression $\lfloor x\rfloor$ is the largest integer that is smaller than or equal to $x$.

\begin{lem}\label{The Inequality}
For any $\omega\in \mathbb{R}_+$, there exist strictly positive constants $r_1$ and $\kappa$ that only depend on $N$, $\alpha_i$'s and $\BU{H}$, such that for every $r>r_1$ we have
$$
\mathbb{E} (Z^{M(r+1)})\geq
\frac{\kappa^M}{r^{M(N+1)}}
\bigl(\frac{M}{\omega^{\sum_{k=1}^{N}\alpha_k}}\bigr)^{rM}
\omega^{rM \textrm{tr}(\BU{H})}
\Bigl(\mathbb{E}( Z^{r})\Bigr)^M.
$$
where $M:=\prod_{i=1}^N \lfloor \omega^{\alpha_i}\rfloor$.
\end{lem}

\begin{proof}

\noindent
\textbf{Step 1: }

%
Let $\mu:=M(r+1)$.
Using proposition \ref{existence and approximation}, we have the following probabilistic representation
\begin{equation}\label{S1E1}
\mathbb{E} (Z^{M(r+1)})
=\mathbb{E}^{\B{\tau}}\bigl(\U{K}_\mu(\B{\tau}_1, \cdots, \B{\tau}_\mu)\bigr),
%
\end{equation}
where $\{\B{\tau}_i\}_{i=1}^\mu$ is a family independent identically distributed (iid) random variables, each $\B{\tau}_i$ being uniformly distributed over $[0,1]^N$, and $\mathbb{E}^{\B{\tau}}$ denotes expectation with respect to the family of random variables $\{\B{\tau}_i\}_{i=1}^\mu$.
%

We define the set of points $\mathcal{P}\subset [0,1]^N$ as
$$
\mathcal{P}:=\Bigl\{\,(\frac{i_1}{\omega^{\alpha_1}}, \frac{i_2}{\omega^{\alpha_2}}, \cdots, \frac{i_N}{\omega^{\alpha_N}})\,\Big|\, i_k\in \{0, 1, \cdots, \lfloor \omega^{\alpha_k}\rfloor-1\}:\, \forall k=1, 2, \cdots, N\Bigr\},
$$
where $\lfloor\cdot\rfloor$ denotes the floor function, as usual.
For every $\B{p}\in\mathcal{P}$, we define the sub-cube $\mathcal{C}^{\circ}(\B{p},\omega^{-\B{\alpha}}) $ as
$$
\mathcal{C}^{\circ}(\B{p},\omega^{-\B{\alpha}}) :=
\B{p}+\prod_{k=1}^{N}(0,\frac{1}{\omega^{\alpha_k}}).
$$
We define the event $\Xi$ as the event on the family $\{\B{\tau}_i\}_{i=1}^{M(r+1)}$ such that every sub-cube $\mathcal{C}^{\circ}(\B{p},\omega^{-\B{\alpha}})$, $\B{p}\in\mathcal{P}$, contains exactly $r+1$ points of the set $\{\B{\tau}_i\}_{i=1}^{M(r+1)}$.
We also define $\Omega^{\mu}_{M}$ as the set of all functions $\theta:\{1, 2, \cdots, M(r+1)\}\rightarrow \mathcal{P}$ such that at every point $\B{p}\in\mathcal{P}$, the cardinality of the inverse image of $\theta$ equals $r+1$, i.e., $|\theta^{-1}(\B{p})|=r+1$; in other words, every member $\theta\in \Omega^{\mu}_{M}$ is a partitioning of the set $\{1, 2, \cdots, M(r+1)\}$ into $M$ distinct boxes indexed by $\mathcal{P}$ such that every box contains exactly $r+1$ elements of $\{1, 2, \cdots, M(r+1)\}$. For every $\theta\in \Omega^{\mu}_{M}$, let $\Xi_\theta$ be the event that the points of $\{\B{\tau}_i\}_{i=1}^{M(r+1)}$ are distributed among the sub-cubes $\{\mathcal{C}^{\circ}(\B{p},\omega^{-\B{\alpha}}) \}_{\B{p}\in\mathcal{P}}$ according to $\theta$. It is clear that $\Xi=\bigcup_{\theta\in \Omega^{\mu}_{M}}\Xi_\theta$. So by Equation \ref{S1E1}, we have
$$
\mathbb{E} (Z^{M(r+1)})\geq
\mathbb{E}^{\B{\tau}}\bigl[\U{K}_\mu(\B{\tau}_1, \cdots, \B{\tau}_\mu)\,\mathbf{1}_{\Xi}\bigr]
=
\sum_{\theta\in \Omega^{\mu}_{M}}\mathbb{E}^{\B{\tau}}\bigl[\U{K}_\mu(\B{\tau}_1, \cdots, \B{\tau}_\mu)\,\mathbf{1}_{\Xi_\theta}\bigr].
$$
For any $\theta\in \Omega^{\mu}_{M}$ and $\B{p}\in\mathcal{P}$,  we use the following notation
$$
\U{K}_{r+1}(\tau, \theta, \B{p}):=\U{K}_{r+1}(\B{\tau}_{i_1}, \cdots, \B{\tau}_{i_{r+1}}),
$$
where $i_1$, ..., $i_{r+1}$ denote all the distinct elements of $\theta^{-1}(\B{p})$.
Using Proposition \ref{Ineuqlity for conditional variance}, for every $\theta \in \Omega^{\mu}_{M}$ we have the following inequality
$$
\U{K}_{M(r+1)}(\B{\tau}_{1}, \cdots, \B{\tau}_{{M(r+1)}})\geq \prod_{\B{p}\in\mathcal{P}}
\U{K}_{r+1}(\tau, \theta, \B{p}).
$$
Hence we obtain
$$
\mathbb{E} (Z^{M(r+1)})\geq
\sum_{\theta\in \Omega^{\mu}_{M}}\mathbb{E}^{\B{\tau}}\bigl[\prod_{\B{p}\in\mathcal{P}}
\U{K}_{r+1}(\tau, \theta, \B{p})\,\mathbf{1}_{\Xi_\theta}\bigr].
$$
For any $\theta\in \Omega^{\mu}_{M}$ and $\B{p}\in\mathcal{P}$, let $\Xi^{\B{p}}_\theta$ denote the event that the points $\B{\tau}_{i_1}, \cdots, \B{\tau}_{i_{r+1}}$ lie in the sub-cube $\mathcal{C}^{\circ}(\B{p},\omega^{-\B{\alpha}})$, where $\{i_1,\cdots,i_{r+1}\}:=\theta^{-1}(\B{p})$. It is evident that
$$ \mathbf{1}_{\Xi_\theta}=\prod_{\B{p}\in\mathcal{P}}\mathbf{1}_{\Xi^{\B{p}}_\theta}.
$$
So we get
\begin{equation}\label{StepOneResult}
\mathbb{E} (Z^{M(r+1)})\geq
\sum_{\theta\in \Omega^{\mu}_{M}}\prod_{\B{p}\in\mathcal{P}}\mathbb{E}^{\B{\tau}}\bigl[
\U{K}_{r+1}(\tau, \theta, \B{p})\,\mathbf{1}_{\Xi^{\B{p}}_\theta}\bigr].
\end{equation}

\noindent
\textbf{Step 2: }
For any $\theta$ and $\B{p}$ fixed, let $\{i_1, \cdots, i_{r+1}\}:=\theta^{-1}(\B{p})$, i.e., $\{i_1, \cdots, i_{r+1}\}$ is the set of indices such that $\B{\tau}_{i_1}, \cdots, \B{\tau}_{i_{r+1}}\in\mathcal{C}^{\circ}(\B{p},\omega^{-\B{\alpha}}) $.  Using Proposition \ref{Reduction Inequality for detCov} we have
$$
\begin{aligned}
&\detcov (\B{X}_{\B{\tau}_{i_{1}}}, \cdots, \B{X}_{\B{\tau}_{i_{r+1}}})\\
&\quad \quad\leq \detcov(\B{X}_{\B{\tau}_{i_{r+1}}})\,
\detcov (\B{X}_{\B{\tau}_{i_{1}}}-\B{X}_{\B{\tau}_{i_{r+1}}}, \cdots, \B{X}_{\B{\tau}_{i_{r}}}-\B{X}_{\B{\tau}_{i_{r+1}}}).
\end{aligned}
$$
By Property $\mathfrak{A}_0$ and using Corollaries \ref{detCov formula} and \ref{Ineuqlity for conditional variance}, we have
$$
\detcov(\B{X}_{\B{\tau}_{i_{r+1}}})\leq c_0^d.
$$
Using the fact that $\B{X}$ has stationary increments, i.e., Property $\mathfrak{A}_2$, we have
$$
\detcov (\B{X}_{\B{\tau}_{i_{1}}}-\B{X}_{\B{\tau}_{i_{r+1}}}, \cdots, \B{X}_{\B{\tau}_{i_{r}}}-\B{X}_{\B{\tau}_{i_{r+1}}})=
\detcov (\B{X}_{\B{\tau}_{i_{1}}-\B{\tau}_{i_{r+1}}}, \cdots, \B{X}_{\B{\tau}_{i_{r}}-\B{\tau}_{i_{r+1}}}).
$$
So we have
$$
\begin{aligned}
\U{K}_{r+1}(\tau, \theta, \B{p})&=
\U{K}_{r+1}(\B{\tau}_{i_{1}}, \cdots, \B{\tau}_{i_{r+1}})\\
&\geq (2\pi)^{-\frac{d(r+1)}{2}}c_0^{-\frac{d}{2}}\,
\bigl(\detcov (\B{X}_{\B{\tau}_{i_{1}}-\B{\tau}_{i_{r+1}}}, \cdots, \B{X}_{\B{\tau}_{i_{r}}-\B{\tau}_{i_{r+1}}})\bigr)^{-\frac{1}{2}}\\
&=c_1 \,\U{K}_{r}(\B{\tau}_{i_{1}}-\B{\tau}_{i_{r+1}}, \cdots, \B{\tau}_{i_{r}}-\B{\tau}_{i_{r+1}}),
\end{aligned}
$$
where $c_2:=(2\pi c_0)^{-\frac{d}{2}}$.
Therefore, we have
\begin{equation}\label{S2E1}
\begin{aligned}
&\mathbb{E}^{\B{\tau}}\bigl[
\U{K}_{r+1}(\tau, \theta, \B{p})\,\mathbf{1}_{\Xi^{\B{p}}_\theta}\bigr]\\
&\quad\quad
\geq c_1\int_{\B{t}_{1},\cdots,\B{t}_{r+1}\in \mathcal{C}^{\circ}(\B{p},\omega^{-\B{\alpha}}) }\U{K}_{r}(\B{t}_{1}-\B{t}_{r+1}, \cdots, \B{t}_{r}-\B{t}_{r+1})\,\mathrm{d}\B{t}_{1}\cdots\mathrm{d}\B{t}_{r+1}\\
&\quad\quad
=c_1\int_{\B{z}\in \mathcal{C}^{\circ}(\B{0},\omega^{-\B{\alpha}}) }
\int_{\B{t}_{1},\cdots,\B{t}_{r}\in 
\mathcal{C}^{\circ}(\B{0},\omega^{-\B{\alpha}})
}\U{K}_{r}(\B{t}_{1}-\B{z}, \cdots, \B{t}_{r}-\B{z})\,\mathrm{d}\B{t}_{1}\cdots\mathrm{d}\B{t}_{r}\mathrm{d}\B{z},
\end{aligned}
\end{equation}
where we used change of variables in the last line.

For every $\B{z}=(z_1, \cdots, z_N)\in \mathcal{C}^{\circ}(\B{0},\omega^{-\B{\alpha}})$, we define
$$
\zeta_{\B{z}}:=\min_{k}\{(\omega^{-\alpha_k}-z_k)^{\frac{1}{\alpha_k}}\},
$$
and
$$
\B{\tilde{\zeta}}_{\B{z}}:=\zeta_{\B{z}}^{\B{\alpha}}=
(\zeta_{\B{z}}^{\alpha_1}, \cdots, \zeta_{\B{z}}^{\alpha_N}).
$$
For every such $\B{z}$, we introduce the new variables $\{\B{s}_k\}_{k=1,\cdots,r}$ in the following way
$$
\B{\tilde{\zeta}}_{\B{z}}\circ\B{s}_{k}:=\B{t}_{k}-\B{z}:\;\forall k=1, \cdots, r,
$$
where $\circ$ as usual, denotes the Schur product of two vectors, i.e., the vector formed by entry-wise multiplication of the two vectors.
It can be easily verified that for every such $\B{z}$ we have
\begin{equation}\label{S2E2}
\begin{aligned}
&\int_
{\B{t}_{1},\cdots,\B{t}_{r}\in \mathcal{C}^{\circ}(\B{0},\omega^{-\B{\alpha}})}
\U{K}_{r}(\B{t}_{1}-\B{z}, \cdots, \B{t}_{r}-\B{z})\,\mathrm{d}\B{t}_{1}\cdots\mathrm{d}\B{t}_{r}\\
&\quad \quad\geq
\zeta_{\B{z}}^{r(\alpha_1+\cdots+\alpha_N)}
\int_{\B{s}_{1},\cdots,\B{s}_{r}\in (0,1)^N}
\U{K}_{r}(\B{\tilde{\zeta}}_{\B{z}}\circ\B{s}_{1}, \cdots, \B{\tilde{\zeta}}_{\B{z}}\circ\B{s}_{r})\,\mathrm{d}\B{s}_{1}\cdots\mathrm{d}\B{s}_{r}.
\end{aligned}
\end{equation}
On the other hand, by diagonal self-similarity (Property $\mathfrak{A}_3$) and noting Remark \ref{diagonal self-similarity covariance formulation}, we have
$$
\cov (\B{X}_{\B{\tilde{\zeta}}_{\B{z}}\circ\B{s}_{1}}, \cdots, \B{X}_{\B{\tilde{\zeta}}_{\B{z}}\circ\B{s}_{r}})=\BU{\Lambda}_{r}
\cov (\B{X}_{\B{s}_{1}}, \cdots, \B{X}_{\B{s}_{r}})\,\BU{\Lambda}_{r}^\dagger,
$$
where $\Lambda_{r}\in\mathbb{R}^{rN\times rN}$ is the block diagonal matrix consisting of $r$ copies of the matrix $\zeta_{\B{z}}^{\BU{H}}$ on its main diagonal, and filled with zero elsewhere, in other words $\BU{\Lambda}_{r}=\textrm{diag}[\zeta_{\B{z}}^{\BU{H}}, \cdots, \zeta_{\B{z}}^{\BU{H}}]$.
Now we notice the fact that the determinant of the exponential of a matrix equals the exponential of its trace, i.e., $\det(e^{\BU{A}})=e^{\textrm{tr}(\BU{A})}$; see e.g. \cite[ch.2]{Hall2003}. So we get
$$
\det \cov (\B{X}_{\B{\tilde{\zeta}}_{\B{z}}\circ\B{s}_{1}}, \cdots, \B{X}_{\B{\tilde{\zeta}}_{\B{z}}\circ\B{s}_{r}})=
\zeta_{\B{z}}^{2r\,\textrm{tr}(\BU{H})},
$$
which implies that
\begin{equation}\label{S2E3}
\U{K}_{r}(\B{\tilde{\zeta}}_{\B{z}}\circ\B{s}_{1}, \cdots, \B{\tilde{\zeta}}_{\B{z}}\circ\B{s}_{r})=
\zeta_{\B{z}}^{-r\,\textrm{tr}(\BU{H})}
\U{K}_{r}(\B{s}_{1}, \cdots, \B{s}_{r}).
\end{equation}
So, by Equations \ref{S2E1}, \ref{S2E2}, and \ref{S2E3} we have
\begin{equation}\label{S2E4}
\mathbb{E}^{\B{\tau}}\bigl[
\U{K}_{r+1}(\tau, \theta, \B{p})\,\mathbf{1}_{\Xi^{\B{p}}_\theta}\bigr]
\geq
c_1 \,\mathbb{E}( Z^{r})
\int_{\B{z}\in \mathcal{C}^{\circ}(\B{0},\omega^{-\B{\alpha}})}
\zeta_{\B{z}}^{r(\sum_{k=1}^{N}\alpha_k- \textrm{tr}(\BU{H}))}
\mathrm{d}\B{z}.
\end{equation}
where we used the following equality which is a result of Proposition \ref{existence and approximation}
$$
\mathbb{E}( Z^{r})=\int_{\B{s}_{1},\cdots,\B{s}_{r}\in (0,1)^N}
\U{K}_{r}(\B{s}_{1}, \cdots, \B{s}_{r})\,\mathrm{d}\B{s}_{1}\cdots\mathrm{d}\B{s}_{r}.
$$
Now we define $\{x_k\}_{i=1}^{N}$ as $x_k:=1-\omega^{\alpha_k}z_k$. By this change of variables, we have
$$
\zeta_{\B{z}}=\frac{1}{\omega} \min_{k}\{(x_k)^{\frac{1}{\alpha_k}}\}.
$$
So we have
\begin{equation}\label{S2E5}
\int_{\B{z}\in \mathcal{C}^{\circ}(\B{0},\omega^{-\B{\alpha}})}
\zeta_{\B{z}}^{-r(\sum_{k=1}^{N}\alpha_k- \textrm{tr}(\BU{H}))}
\mathrm{d}\B{z}
=\omega^{-\sum_{k=1}^{N}\alpha_k} \omega^{-r(\sum_{k=1}^{N}\alpha_k- \textrm{tr}(\BU{H}))}
J(\B{\alpha},\BU{H},r),
\end{equation}
where
$$
J(\B{\alpha},\BU{H},r):=
\int_0^1\cdots\int_0^1
\bigl(\min_{k}\{(x_k)^{\frac{1}{\alpha_k}}\}\bigl)^{r(\sum_{k=1}^{N}\alpha_k- \textrm{tr}(\BU{H}))}
\,\mathrm{d}x_1\cdots \mathrm{d}x_N.
$$
This shows that if $\sum_{k=1}^{N}\alpha_k$ is smaller than $ \textrm{tr}(\BU{H})$, then property $\mathfrak{A}_1$ can not hold. So we assume $\sum_{k=1}^{N}\alpha_k\geq \textrm{tr}(\BU{H})$.

Denote $\eta:=r(\sum_{k=1}^{N}\alpha_k- \textrm{tr}(\BU{H}))$ and $\alpha_0:=\min \{\alpha_1,\cdots,\alpha_N\}$. One can easily verify that for $r$ larger than $r_1:=\frac{2\alpha_0}{\sum_{k=1}^{N}\alpha_k- \textrm{tr}(\BU{H})}$ (so that $\eta>2\alpha_0$), we have
$$
\begin{aligned}
&\int_0^1\cdots\int_0^1
\bigl(\min_{k}\{(x_k)^{\frac{1}{\alpha_k}}\}\bigl)^{\eta}
\,\mathrm{d}x_1\cdots \mathrm{d}x_N
\geq
\int_0^1\cdots\int_0^1
\bigl(\min_{k}\{x_k\}\bigl)^{\frac{\eta}{\alpha_0}}
\,\mathrm{d}x_1\cdots \mathrm{d}x_N\\
&\quad \quad \geq
\int_{1-\frac{\alpha_0}{\eta}}^1\cdots\int_{1-\frac{\alpha_0}{\eta}}^1
\bigl(\min_{k}\{x_k\}\bigl)^{\frac{\eta}{\alpha_0}}
\,\mathrm{d}x_1\cdots \mathrm{d}x_N
\geq (\frac{\alpha_0}{\eta})^N (1-\frac{\alpha_0}{\eta})^{\frac{\eta}{\alpha_0}}\\
&\quad \quad\geq C (\frac{\alpha_0}{\eta})^N,
\end{aligned}
$$
where $C>0$ is global contact. So for $r$ large enough ($r\geq r_1$), we have
\begin{equation}\label{S2E6}
J(\B{\alpha},\BU{H},r)\geq \frac{c_2}{r^N},
\end{equation}
where $c_2>0$ is a constant that only depends on $\alpha_i$'s and $N$. When $\sum_{k=1}^{N}\alpha_k= \textrm{tr}(\BU{H})$, Equation \eqref{S2E6} remains valid for every $r\in \mathbb{N}$. So in this case we define $r_1$ equal to $1$.

So by applying Equations \eqref{S2E6} and \eqref{S2E5} into Equation \eqref{S2E4} we get
\begin{equation}\label{S2E7}
\mathbb{E}^{\B{\tau}}\bigl[
\U{K}_{r+1}(\tau, \theta, \B{p})\,\mathbf{1}_{\Xi^{\B{p}}_\theta}\bigr]
\geq
\frac{c_1 c_2}{r^N \omega^{\sum_{k=1}^{N}\alpha_k}}
\omega^{-r(\sum_{k=1}^{N}\alpha_k- \textrm{tr}(\BU{H}))}
\mathbb{E}( Z^{r}).
\end{equation}

\noindent
\textbf{Step 3: }
Applying Equation \eqref{S2E7} to \eqref{StepOneResult} we get
\begin{equation}\label{S3E1}
\mathbb{E} (Z^{M(r+1)})\geq
\frac{c_1^{M}c_2^M \mathfrak{N}^{M(r+1)}_{M}}{r^{MN} \omega^{M\sum_{k=1}^{N}\alpha_k}}
\omega^{-rM(\sum_{k=1}^{N}\alpha_k- \textrm{tr}(\BU{H}))}
\Bigl(\mathbb{E}( Z^{r})\Bigr)^M,
\end{equation}
where $\mathfrak{N}^{M(r+1)}_{M}:=|\Omega^{M(r+1)}_{M}|$, i.e., the cardinality of $\Omega^{M(r+1)}_{M}$. Using Lemma \ref{Partitioning a set}, we have
$$
\mathfrak{N}^{M(r+1)}_{M}\geq
\frac{\kappa_1 \sqrt{M}}{\kappa_2^M \sqrt{(r+1)^M}} M^{M(r+1)},
$$
where $\kappa_1$ and $\kappa_2$ are global constants. So we have
$$
\mathbb{E} (Z^{M(r+1)})\geq
\frac{\kappa_1 c_1^{M}c_2^M \sqrt{M}}{\kappa_2^M r^{MN} \sqrt{(r+1)^{M}} }
\bigl(\frac{M}{\omega^{\sum_{k=1}^{N}\alpha_k}}\bigr)^{M(r+1)}
\omega^{rM \textrm{tr}(\BU{H})}
\Bigl(\mathbb{E}( Z^{r})\Bigr)^M.
$$
which clearly implies the statement of the lemma.
\end{proof}


\begin{lem}\label{Iterated Inequality}
There exists a positive number $r_1$ that only depends on $N$, $\alpha_i$'s and $\BU{H}$, such that for every $r>r_1$ and any $\omega\in \mathbb{R}_+$ with the property that $\omega^{\alpha_i}$'s are all integer numbers, and for any positive integer $q$, we have
\begin{equation}\label{L3E1}
\frac{\Bigl(\mathbb{E} (Z^{r M^q(1+o_r)})\Bigr)^{\frac{1}{r M^q(1+o_r)}}}
{(r M^q(1+o_r))^{\lambda}}
\geq
B_{\omega}(r)\,
\Bigl(\frac{
\sqrt[r]{\mathbb{E}(
Z^{r})}
}{r^{\lambda}}\Bigr)^{\frac{1}{1+o_r}},
\end{equation}
where $M:=\prod_{i=1}^{N}\lfloor\omega^{\alpha_i}\rfloor$, $o_r:=\frac{1}{r}\sum_{k=0}^{q-1}\frac{1}{M^k}$, and $B_{\omega}(r)$ is a strictly-positive-valued function ($B_{\omega}(r)>0$) that depends only on $\omega$, $r$, $N$ and $\BU{H}$ such that $\lim_{r\rightarrow+\infty} B_{\omega}(r)=1$.
\end{lem}
\begin{proof}
By Lemma \ref{The Inequality} we have
$$
\mathbb{E} (Z^{M(r+1)})\geq
\frac{\kappa^M}{r^{M(N+1)}}
(\rho
\omega^{ \textrm{tr}(\BU{H})})^{rM}
\Bigl(\mathbb{E}( Z^{r})\Bigr)^M,
$$
where $M=\prod_{i=1}^N \lfloor \omega^{\alpha_i}\rfloor$, and $\rho:=\frac{M}{\omega^{\sum_{k=1}^{N}\alpha_k}}$.
Reiterating this inequality $q$ times, and using the inequality $M^k r+\sum_{i=1}^{k}M^i\leq M^k(r+2)$,
we get
$$
\mathbb{E} (Z^{r M^q+\sum_{i=1}^{q}M^i})\geq
A_{\omega,r,q}\, (\rho \omega^{ \textrm{tr}(\BU{H})})^{(q r M^q+\sum_{i=1}^{q-1}i M^{i+1})}
\Bigl(\mathbb{E}( Z^{r})\Bigr)^{M^q},
$$
where
$$
A_{\omega,r,q}:=\frac{\kappa^{\sum_{i=1}^q M^i}}{M^{(N+1)\sum_{i=1}^{q-1}i M^{q-i}}\, (r+2)^{(N+1)\sum_{i=1}^q M^i}}.
$$
%
For $\lambda:=\frac{\textrm{tr}(\BU{H})}{\sum_{i=1}^N\alpha_i}$, and using the notation $o_r:=\frac{1}{r}\sum_{i=0}^{q-1}\frac{1}{M^i}$, we have
\begin{equation}\label{T1S1E1}
\frac{\Bigl(\mathbb{E} (Z^{r M^q+\sum_{i=1}^{q}M^i})\Bigr)^{\frac{1}{r M^q+\sum_{i=1}^{q}M^i}}}
{(r M^q+\sum_{i=1}^{q}M^i)^{\lambda}}
\geq
B_{\omega,r,q}\,
\Bigl(\frac{
\sqrt[r]{\mathbb{E}(
Z^{r})}
}{r^{\lambda}}\Bigr)^{\frac{1}{1+o_r}},
\end{equation}
where
$$
B'_{\omega,r,q}:=(A_{\omega,r,q})^{\frac{1}{r M^q(1+o_r)}}
(\rho \omega^{ \textrm{tr}(\BU{H})})^{
\frac{q r M^q+\sum_{i=1}^{q-1}i M^{i+1}}{r M^q+\sum_{i=1}^{q}M^i}
} M^{-q\lambda}
r^{-\lambda\frac{o_r}{1+o_r}}
(1+o_r)
^{-\lambda}.
$$
Using the inequalities $\sum_{i=0}^{+\infty} x^{i}=\frac{1}{(1-x)}$ and $\sum_{i=1}^{+\infty}i x^{i-1}=\frac{1}{(1-x)^2}$, and noting that $M\geq2$, we can easily verify that as $r$ goes to $+\infty$, the function $o_r$ converges to zero uniformly in $q$ and $\omega$, and
\begin{equation}
M^{q-1}\leq \sum_{i=1}^{q-1}i M^{q-i} \leq 4 M^{q-1}
\quad \text{and} \quad
M^{q}\leq \sum_{i=1}^{q}i M^{i} \leq 2 M^{q}.
\end{equation}
Using these inequalities, we can easily show that
\begin{equation}\label{T1S1E2}
(A_{\omega,r,q})^{\frac{1}{r M^q(1+o_r)}}\geq A_r,
\end{equation}
where $A_r>0$ is only a function of $r$ and $N$ such that $\lim_{r\rightarrow+\infty} A_r=1$.
It is also easy to verify that
$$
q\sum_{i=1}^{q}M^i-\sum_{i=1}^{q-1}i M^{i+1}=M^q \sum_{i=1}^{q} \frac{i}{M^{i-1}},
$$
and hence
\begin{equation}\label{T1S1E3}
q-\frac{q r M^q+\sum_{i=1}^{q-1}i M^{i+1}}{r M^q+\sum_{i=1}^{q}M^i}=\frac{\sum_{i=1}^{q} \frac{i}{M^{i-1}}}{r(1+o_r)}\leq \frac{1}{r(1+o_r)(1-\frac{1}{M})^2}\,.
\end{equation}
Noting that under the assumptions of the lemma, $\rho$ is equal to $1$, and using Equations \eqref{T1S1E2} and \eqref{T1S1E3}, we obtain
$$
B'_{\omega,r,q}\geq B_{\omega}(r),
$$
where $B_{\omega}(r)$ is a strictly-positive-valued function that only depends on $N$, $r$, $\omega$, and $\BU{H}$ such that
$\lim_{r\rightarrow+\infty} B_{\omega}(r)=1$.
This completes the proof.
\end{proof}

\begin{lem}\label{existence of moments limit}
Let $\B{X}_{\B{t}}$ be an (N,d)-Gaussian random field satisfying Properties $\mathfrak{A}_1$, $\mathfrak{A}_0$, $\mathfrak{A}_2$, and $\mathfrak{A}_3$ with self-similarity vector $\B{\alpha}:=(\alpha_1,\cdots,\alpha_N)$ such that for every $i$ and $j$, the quotient $\alpha_i/\alpha_j$ is a rational number. Then the following limit exists
\begin{equation}\label{Theorem Limit Equation}
\lim_{n\rightarrow +\infty}\frac{\sqrt[n]{\mathbb{E}(Z^n)}}{n^\lambda},
\end{equation}
where $Z:=L_{\B{0}}(\B{X},[0,1]^N)$ and $\lambda:=\frac{ \textrm{tr}(\BU{H})}{\sum_{k=1}^{N}\alpha_k}$.
\end{lem}
\begin{proof}

Define
$$
\overline{\ell}:=\limsup_{n\rightarrow+\infty} \frac{
\sqrt[n]{\mathbb{E}(Z^{n})}}{n^{\lambda}}
\quad \text{and} \quad
\underline{\ell}:=\liminf_{n\rightarrow+\infty} \frac{
\sqrt[n]{\mathbb{E}(Z^{n})}}{n^{\lambda}}.
$$
Consider any positive real number $\ell$ that is strictly less than $\overline{\ell}$. Let $\ell_1$ and $\ell_2$ be real numbers satisfying $\ell<\ell_1<\ell_2<\overline{\ell}$.

\noindent
\textbf{Step 1: }
As for every $i$ and $j$, the quotient $\alpha_i/\alpha_j$ is a rational number, we can find a real number $\alpha>0$ such that for every $i$, the quotient $\frac{\alpha_i}{\alpha}$ is an integer. Now choose $\omega_1$ and $\omega_2$ such that $\omega_1^\alpha=2$ and $\omega_2^\alpha=3$. Clearly, in this case all $\omega_j^{\alpha_i}$'s are integer-valued for every $j=1,2$ and $i=1,\cdots,N$, hence we may apply Lemma \ref{Iterated Inequality}. Also note that in this case, there exists a positive integer $m_0$ such that $M_1:=\prod_{i=1}^{N}\lfloor\omega_1^{\alpha_i}\rfloor=2^{m_0}$ and $M_2:=\prod_{i=1}^{N}\lfloor\omega_2^{\alpha_i}\rfloor=3^{m_0}$.
Let $r$ be any integer larger than $r_1$, and $p$ and $q$ be two arbitrary positive integers. Applying Equation \eqref{L3E1} first with $\omega_1$ and $p$, and then repeating it with $\omega_2$ and $q$, we get
\begin{equation}\label{T2S1E1}
\frac{\Bigl(\mathbb{E} (Z^{\Phi_{r,p,q}})\Bigr)^{\frac{1}{\Phi_{r,p,q}}}}
{(\Phi_{r,p,q})^{\lambda}}
\geq
B_{\omega_2}(R)\,\bigl({B_{\omega_1}(r)}\bigr)^{\frac{1}{1+o_R}}
\Bigl(\frac{
\sqrt[r]{\mathbb{E}(
Z^{r})}
}{r^{\lambda}}\Bigr)^{\frac{1}{(1+o_r)(1+\bar{o}_R)}}.
\end{equation}
where $R:=r M_1^p+\sum_{k=1}^{p}M_1^k$, $o_r:=\frac{1}{r}\sum_{k=0}^{p-1}\frac{1}{M_1^k}$, $\bar{o}_R:=\frac{1}{R}\sum_{k=0}^{q-1}\frac{1}{M_2^k}$, and
$$
\Phi_{r,p,q}:=R M_2^q+\sum_{k=1}^{q}M_2^k=r 2^{m_0p} 3^{m_0q} (1+o_R)(1+o_r).
$$
We note that $B_{\omega_1}$ and $B_{\omega_2}$ converge to one uniformly in $p$ and $q$, and $o_r$ and $o_R$ converge to zero uniformly in $p$ and $q$.

\noindent
\textbf{Step 2: }
Choose $\varepsilon>0$ such that $(1+\varepsilon)<\frac{\ell_1}{\ell}$. 
Clearly, there exists $r_2>0$ such that for every $R, r\geq r_2$ we have
$$
(1+o_r)(1+\bar{o}_R)<1+\varepsilon\quad\text{, and}\quad B_{\omega_2}(R)\,\bigl({B_{\omega_1}(r)}\bigr)^{\frac{1}{1+o_R}}
\ell_2^{\frac{1}{(1+o_r)(1+o_R)}}>\ell_1.
$$
By the definition of $\limsup$,  there exists an integer $r>\max\{r_1,r_2\}$ such that $\frac{\sqrt[r]{\mathbb{E}(Z^{r})}}{r^{\lambda}}>\ell_2$. Now we apply this $r$ to Equation \eqref{T2S1E1}, along with any arbitrary integers $p$ and $q$. Noting that $R=r M_1^p+\sum_{k=1}^{p}M_2^k>r>r_2$, we have
\begin{equation}\label{T2S1E2}
\frac{\Bigl(\mathbb{E} (Z^{\Phi_{r,p,q}})\Bigr)^{\frac{1}{\Phi_{r,p,q}}}}
{(\Phi_{r,p,q})^{\lambda}}
\geq \ell_1 \;;\quad \forall p,q\in \mathbb{N}.
\end{equation}
We also have
\begin{equation}\label{T2S1E3}
r 2^{m_0p} 3^{m_0q}\leq \Phi_{r,p,q}=r 2^{m_0p} 3^{m_0q} (1+o_R)(1+o_r)\leq r 2^{m_0p} 3^{m_0q} (1+\varepsilon).
\end{equation}
As $\log_{2}3$ is not a rational number, by  Dirichlet's approximation theorem (Theorem \ref{Dirichlet}) there exist $p_0,q_0\in \mathbb{N}$ such that
\begin{equation}\label{T2S1E4}
0<|p_0-q_0\log_23|<\frac{1}{m_0} \log_2 (\frac{\ell_1}{\ell(1+\varepsilon)}).
\end{equation}
We proceed with the assumption that $p_0>q_0\log_23$; when $p_0<q_0\log_23$, the proof is similar.
So by Equation \eqref{T2S1E4} we have
\begin{equation}\label{T2S1E41}
1<\nu:=\frac{2^{m_0p_0}}{3^{m_0q_0}}<\frac{\ell_1}{\ell(1+\varepsilon)}.
\end{equation}
We choose $k_0\in \mathbb{N}$ such that $\nu^{k_0}>3$, and define $n_1:=r(1
+\varepsilon) 3^{m_0 q_0 k_0}$. Take any arbitrary integer $n\geq n_1$, and define the following
$$
q_1:=\max\{k; \; n\geq r (1+\varepsilon) 3^{m_0 k}\}\quad
\text{and}
\quad
k_1:=\max\{k; \; n\geq r (1+\varepsilon) 3^{m_0 q_1}\nu^{k}\}.
$$
As we have
$$
3^{m_0 q_1}\nu^{k_1}=3^{m_0 (q_1-k_1 q_0)} 2^{m_0 k_1 q_0},
$$
hence
\begin{equation}\label{T2S1E5}
r(1+\varepsilon)3^{m_0 (q_1-k_1 q_0)} 2^{m_0 k_1 q_0}\leq n <\nu r(1+\varepsilon) 3^{m_0 (q_1-k_1 q_0)} 2^{m_0 k_1 q_0}.
\end{equation}
We note that $q=q_1-k_1 q_0\geq 0$, because \\
(I) as $n>n_1$, by definition $q_1\geq q_0 k_0$, and \\
(II) $k_1\leq k_0$, otherwise if $k_0< k_1$, then $\nu^{k_1}>3$, and hence $n>r(1+\varepsilon)3^{m_0 (q_1+1)}$ which is in contradiction with the definition of $q_1$.\\
So we can apply Equation \eqref{T2S1E2} to $q=q_1-k_1 q_0$ and $p=k_1 q_0$. By Equations \eqref{T2S1E3} and \eqref{T2S1E5}, we get
\begin{equation}\label{T2S1E6}
\Phi_{r,p,q}\leq n \leq \nu (1+\varepsilon) \Phi_{r,p,q}.
\end{equation}

\noindent
\textbf{Step 3: }
As $\Phi_{r,p,q}\leq n$, by H\"older's inequality we have
$$
\Bigl(\mathbb{E} (Z^{\Phi_{r,p,q}})\Bigr)^{\frac{1}{\Phi_{r,p,q}}}\leq
\Bigl(\mathbb{E} (Z^{n})\Bigr)^{\frac{1}{n}}.
$$
Hence, by Equation \eqref{T2S1E6} we have
\begin{equation}\label{T2S3E1}
\bigl(\nu(1+\varepsilon)\bigr)^{-\lambda}\frac{\Bigl(\mathbb{E} (Z^{\Phi_{r,p,q}})\Bigr)^{\frac{1}{\Phi_{r,p,q}}}}
{(\Phi_{r,p,q})^{\lambda}}\leq
\frac{\Bigl(\mathbb{E} (Z^{n})\Bigr)^{\frac{1}{n}}}
{n^{\lambda}}.
\end{equation}
But by \eqref{T2S1E41}, we have
$$
\frac{\ell}{\ell_1}\leq
\bigl(\nu(1+\varepsilon)\bigr)^{-1}\leq \bigl(\nu(1+\varepsilon)\bigr)^{-\lambda}.
$$
So by Equations \eqref{T2S3E1} and \eqref{T2S1E2} we finally get
$$
\frac{\Bigl(\mathbb{E} (Z^{n})\Bigr)^{\frac{1}{n}}}
{n^{\lambda}}\geq \ell.
$$
This means that $\underline{\ell}$ is larger than or equal to $\ell$. As
this is true for any positive number $\ell$ that is strictly less than $\overline{\ell}$, this implies $\overline{\ell}=\underline{\ell}$; in other words the limit in \eqref{Theorem Limit Equation} exists.
\end{proof}

\begin{proof}[Proof of Theorem \ref{main theorem on local times}]
Lemma \ref{existence of moments limit} guarantees the convergence of $\{\frac{\sqrt[n]{\mathbb{E}(Z^n)}}{n^\lambda}\}_n$, so Theorem  \ref{Kasahara} can be applied if we show that the limit is strictly positive. This can indeed be easily verified applying Lemma \ref{Iterated Inequality} with some arbitrary $r>r_1$ and $\omega\in \mathbb{R}_+$ such that $\omega^{\alpha_i}$'s are all integers, and then letting $q$ converge to $+\infty$. Clearly the left-hand side of Equation \eqref{L3E1} converges to $\lim_{n\rightarrow +\infty}\frac{\sqrt[n]{\mathbb{E}(Z^n)}}{n^\lambda}$ whereas the right-hand side is strictly positive and independent of $q$.
\end{proof}

%

Now we can easily prove Corollary
\ref{limit theorem on intersection local times}.

\begin{proof}[Proof of Corollary \ref{limit theorem on intersection local times}]

We define the Gaussian field $\B{\Delta}_{\B{\tilde{t}}}$
$$
\B{\Delta}_{\B{\tilde{t}}}:=\bigl(\B{X}_1(\B{t_1})-\B{X}_2(\B{t_2}), \B{X}_2(\B{t_2})-\B{X}_3(\B{t_3}),\cdots, \B{X}_{m-1}(\B{t_{m-1}})-\B{X}_m(\B{t_m})\bigr),
$$
where $\B{\tilde{t}}:=(\B{\tilde{t}}_1, \cdots, \B{\tilde{t}}_n)$ and $\B{\tilde{t}}_1\in \mathcal{I}^{\tilde{N}_k}$ for every $k=1,\cdots,m$.

It is evident that $\B{\Delta}_{\B{\tilde{t}}}$ is a a centered Gaussian $(\tilde{N},\tilde{d})$-field, where $\tilde{d}:=(m-1)d$ and $\tilde{N}:=\sum_{k=1}^{m}N_k$.

The proof of the existence of the local time of $\B{\Delta}_{\B{\tilde{t}}}$ around $\B{0}$ over the cube $\mathcal{I}^{\tilde{N}}$ and the finiteness of all its moments, is similar to the proof of Theorem \ref{Upper bound for intersection local times}. Indeed, using Equation \eqref{upper bound on intersection kernel} with $q_i=\frac{m-1}{m}$ ($p_i=\frac{1}{m}$) for every $i=1,\cdots,m$, we obtain
$$
\U{K}_n^{\B{\Delta}}(\B{\tilde{t}}_1, \cdots, \B{\tilde{t}}_n)\leq
\prod_{k=1}^{m}\bigl(\U{K}_n^{\B{X}_k}(\B{t}_k^1, \cdots, \B{t}_k^n)\bigr)^{\frac{m-1}{m}}.
$$
So the Gaussian field $\B{\Delta}_{\B{\tilde{t}}}$ satisfies Property $\mathfrak{A}_1$.

%
%

As all the random fields $\B{X}_k$, $k=1,\cdots,m$, satisfy Properties $\mathfrak{A}_0$ and $\mathfrak{A}_2$, so does the Gaussian field $\B{\Delta}_{\B{\tilde{t}}}$. As every $\B{X}_k$ is diagonally self-similar (i.e., it satisfies Property $\mathfrak{A}_3$) with scaling vector $\B{\alpha}_k:=(\alpha_{k,1}, \cdots, \alpha_{k,N_k})\in \mathbb{R}_+^{N_k}$ and scaling matrix $\BU{H}\in\mathbb{R}^{d\times d}$, it can be easily verified that $\B{\Delta}_{\B{\tilde{t}}}$ is also diagonally self-similar with the scaling vector $\B{\tilde{\alpha}}\in\mathbb{R}_+^{\tilde{N}}$ constructed by adjoining all the vectors $\B{\alpha}_k$ together, i.e., $\B{\tilde{\alpha}}:=(\B{\alpha}_1, \B{\alpha}_2, \cdots, \B{\alpha}_m)$ and with the scaling matrix $\BU{\tilde{H}}\in\mathbb{R}^{(m-1)d\times(m-1)d}$ which is a block diagonal matrix containing $m-1$ copies of $\BU{H}$ on its main diagonal and zero elsewhere; in other words, $\BU{\tilde{H}}:=\textrm{diag}(\BU{H}, \BU{H}, \cdots, \BU{H})$. Clearly in this case we have $\textrm{tr}(\BU{\tilde{H}})=(m-1)\textrm{tr}\BU{H}$.
Now the desired conclusion is evident applying Theorem \ref{main theorem on local times}.

\end{proof}

\section{Appendix}
\begin{lem}\label{conditional variance formula}
Let $\mathcal{H}$ be a Gaussian Hilbert space, i.e., for any $n\in\mathbb{Z}^{+}$, and any elements $X_1, \cdots, X_n\in \mathcal{H}$, the set $\{X_i\}_{i=1}^n$ is a family of jointly Gaussian zero-mean random variables, and $\mathcal{H}$ forms a Hilbert space with respect to the inner product $\langle X,Y\rangle:=\mathbb{E}(X Y)$. Let $\mathcal{G}$ be a subspace of $\mathcal{H}$, and $X$ be an element of $\mathcal{H}$. Then we have
$$
\var(X\big|\mathcal{G})=\|Q_{\mathcal{G}}(X)\|^2,
$$
where $Q_{\mathcal{G}}(X):=X-P_{\mathcal{G}}(X)$, and $P_{\mathcal{G}}(X)$ is the orthogonal projection of $X$ over the subspace $\mathcal{G}$.
\end{lem}
\begin{proof}
By definition, we have
$$
\var(X\big|\mathcal{G})=
\mathbb{E}\bigl[\bigl(X-\mathbb{E}(X\big|\mathcal{G})\bigr)^2\big |\mathcal{G}\bigl].
$$
Replacing $X$ by $P_{\mathcal{G}}(X)+Q_{\mathcal{G}}(X)$ on the right-hand side of the above equation, and noting that $Q_{\mathcal{G}}(X)$ is independent of $\mathcal{G}$, we can easily derive the desired result.
\end{proof}
\begin{cor}\label{conditionging decreases variance}
An immediate implication of the previous lemma is the following inequality
$$
\var(X\big|\mathcal{G})\leq \var(X).
$$
\end{cor}

\begin{lem}\label{Ineuqlity for conditional variance}
Let $\mathcal{Y}$ be an arbitrary inner-product space. Then for any $\B{y}_1,\cdots,\B{y}_n \in \mathcal{Y}$, $n\in\mathbb{N}$, we have
$$
\det
\begin{bmatrix}
\B{y}_1\\
\vdots\\
\B{y}_n
\end{bmatrix}
\begin{bmatrix}
\B{y}_1&\cdots&\B{y}_n
\end{bmatrix}=\|\B{y}_1\|^2
\prod_{k=2}^{n} \|Q_{\langle \B{y}_1, \cdots, \B{y}_{k-1}\rangle}(\B{y}_k)\|^2,
$$
where $\langle \B{y}_1, \cdots, \B{y}_{k-1}\rangle$ is the subspace generated by $\{\B{y}_1, \cdots, \B{y}_{k-1}\}$.
\end{lem}
\begin{proof}
We assume that $\{\B{y}_i\}_{i=1}^n$ are linearly independent, because otherwise, the equality is trivially true. By orthogonal decomposition of each $\B{y}_k$ over the subspace $\langle \B{y}_1, \cdots, \B{y}_{k-1}\rangle$, we can obtain the sequences $\{\B{f}_i\}_{i=1}^n$,  $\{\eta_i\}_{i=1}^n$, and $\{\B{p}_i\}_{i=1}^n$ such that for $\forall k=1, \cdots, n$, we have $\B{y}_k=\eta_k \B{f}_k+\B{p}_k$ where $\eta_k\in \mathbb{R}_+$, $\B{p}_k\in \langle \B{y}_1, \cdots, \B{y}_{k-1}\rangle$, $\B{f}_k\perp \langle \B{y}_1, \cdots, \B{y}_{k-1}\rangle$, and ${\|\B{f}_i\|=1}$. In fact for each $k=2, \cdots, n$, $\B{p}_k=P_{\langle \B{y}_1, \cdots, \B{y}_{k-1}\rangle}(\B{y}_k)$ and $\eta_k \B{f}_k=Q_{\langle \B{y}_1, \cdots, \B{y}_{k-1}\rangle}(\B{y}_k)$.
Let $f:\mathbb{R}^n\rightarrow \mathcal{Y}$ be the linear isometry such that $f(\B{e}_i)=\B{f}_i$ for every $i=1, \cdots, n$, where $\{\B{e}_i\}_{i=1}^n$ is the standard basis for the Euclidean space $\mathbb{R}^n$, i.e., $\B{e}_i$ is the \underline{column vector} that is $1$ in the $i$-th entry and $0$ elsewhere.
For each $i=1, \cdots,n$, define $\B{x}_i\in\mathbb{R}^n$ as the inverse image of $\B{y}_i$, i.e., $f(\B{x}_i)=\B{y}_i$. As $f$ is an isometry, we have
$$
\det
\begin{bmatrix}
\B{y}_1\\
\vdots\\
\B{y}_n
\end{bmatrix}
\begin{bmatrix}
\B{y}_1&\cdots&\B{y}_n
\end{bmatrix}=
\det
\begin{bmatrix}
\B{x}_1^T\\
\vdots\\
\B{x}_n^T
\end{bmatrix}
\begin{bmatrix}
\B{x}_1&\cdots&\B{x}_n
\end{bmatrix}=
\Bigl(\det
\begin{bmatrix}
\B{x}_1&\cdots&\B{x}_n
\end{bmatrix}\Bigr)^2.
$$
Again due to the fact that $f$ is an isometry, for every $k$ we have $\B{x}_k\in\langle \B{e}_1, \cdots, \B{e}_{k} \rangle$, i.e., the matrix $\begin{bmatrix} \B{x}_1&\cdots&\B{x}_n \end{bmatrix}$ is upper triangular with $\eta_i$'s on its diagonal. So we have
$$
\det
\begin{bmatrix}
\B{x}_1&\cdots&\B{x}_n
\end{bmatrix}
=\prod_{i=1}^n \eta_i.
$$
But $\eta_1=\|\B{y}_1\|$, and $\eta_i=\|Q_{\langle \B{y}_1, \cdots, \B{y}_{i-1}\rangle}(\B{y}_i)\|$ for every $i=2, \cdots, n$. So the proof is complete.
\end{proof}

\begin{cor}\label{detCov formula}
Suppose $\{X_i\}_{i=1}^{n}$ is a family of jointly Gaussian random variables. Using Lemma \ref{conditional variance formula}, we obtain the following formula
$$
\det \cov (X_1, \cdots, X_n)=\var(X_1) \prod_{k=2}^{n} \var(X_k\big| X_1, \cdots, X_{k-1})
$$
\end{cor}
Using Lemma \ref{Ineuqlity for conditional variance} we can easily verify the following two propositions.
\begin{prop}\label{Inequality for detCov}
Suppose that $\{Y_i^j;\; i=1, \cdots,n, j=1,\cdots, m_i\}$ is family of jointly Gaussian random variables, where $m_i\in\mathbb{N}$ for every $i=1, \cdots n$. Also, for each $i=1, \cdots n$, let $\B{Y}_i:=(Y_i^1,\cdots, Y_i^{m_i})$. Then we have
$$
\begin{aligned}
\detcov(\B{Y}_1, \cdots, \B{Y}_n)
=\detcov(\B{Y}_1) \prod_{k=2}^{n} \detcov(\B{Y}_k\big| \B{Y}_1, \cdots, \B{Y}_{k-1})
\leq \prod_{i=1}^{n} \detcov(\B{Y}_i),
\end{aligned}
$$
where $\detcov(\B{Y}_{k}\big| \B{Y}_1, \cdots, \B{Y}_{k-1})$ is the determinant of the conditional covariance matrix of $\B{Y}_{k}$ conditioned on the random vectors $\B{Y}_1$, ..., $\B{Y}_{k-1}$.
\end{prop}

\begin{prop}\label{Reduction Inequality for detCov}
Let $m\in\mathbb{N}$, and consider any family of $n$ jointly Gaussian random vectors of size $m$, i.e.,
$\BU{Y}_i:=(Y_i^1,\cdots, Y_i^{m})$ for every $i$. Then we have
$$
\begin{aligned}
&\detcov(\BU{Y}_1, \cdots, \BU{Y}_{n})\\
&\quad\quad \leq \detcov(\BU{Y}_{k}) \detcov(\BU{Y}_1-\BU{Y}_{k}, \cdots, \BU{Y}_{k-1}-\BU{Y}_{k}, \BU{Y}_{k+1}-\BU{Y}_{k}, \cdots, \BU{Y}_{n}-\BU{Y}_{k})
\end{aligned}
$$
\end{prop}

We have the following lemma which can be proved by elementary probability and then Stirling's approximation, i.e., the fact that
$\kappa_1\leq \frac{n!}{(\frac{n}{e})^n \sqrt{n}}\leq \kappa_2$, where $\kappa_1$ and $\kappa_2$ are strictly positive global constants.
\begin{lem}\label{Partitioning a set}
Let $n,m\in\mathbb{N}$, and suppose that we have $nm$ distinct balls and $n$ distinct baskets. Let $\mathfrak{N}^{nm}_{n}$ be the number of different ways one can distribute the balls among the baskets such that each basket contains exactly $m$ balls. In other words, $\mathfrak{N}^{nm}_{n}$ is the cardinality of $\Omega^{mn}_{n}$ where $\Omega^{mn}_{n}$ is the set of all functions $\sigma:\{1, 2, \cdots, mn\}\rightarrow \{1, 2, \cdots, n\}$ such that for every $p\in\{1, 2, \cdots, n\}$, the cardinality of the inverse image of $\sigma$ equals $m+1$, i.e., $|\sigma^{-1}(\B{p})|=m$. We have
$$
\mathfrak{N}^{nm}_{n}=\frac{(nm)!}{(m!)^n}
\geq
\frac{}{} \frac{\kappa_1 \sqrt{n}}{\kappa_2^n \sqrt{m^n}} n^{nm},
$$
where $\kappa_1$ and $\kappa_2$ are strictly positive global constants.
\end{lem}

The following theorem, also known as Dirichlet's theorem on Diophantine approximation, is a direct application of Pigeonhole Principle which itself was first used by Dirichlet \cite{Dirichlet1863}. For completeness we provide the proof.
\begin{thm}[Dirichlet's approximation theorem]\label{Dirichlet}
For any real number $\alpha$ and any positive integer $n$, there exist integers $p$ and $q$ such that $1\leq q \leq n$ and
$
|p-q \alpha|\leq \frac{1}{n}
$
\end{thm}
\begin{proof}
Consider the numbers $\alpha-\lfloor\alpha\rfloor$, $2\alpha-\lfloor2\alpha\rfloor$, ..., $n\alpha-\lfloor n\alpha\rfloor$, and the intervals $[\frac{i}{n},\frac{i+1}{n})$, for $i=0,\cdots,n-1$. Either one of the numbers falls into the first interval  $[0,\frac{1}{n})$, or otherwise there will an interval that contains more than one point. In either case we can find the desired $p$ and $q$.
\end{proof}


\bibliographystyle{plain}

\begin{thebibliography}{10}

\bibitem{ChenLiRosinskyShao11}
Xia Chen, Wenbo~V. Li, Jan Rosi\'nski, and Qi-Man Shao.
\newblock Large deviations for local times and intersection local times of
  fractional {B}rownian motions and {R}iemann-{L}iouville processes.
\newblock {\em Ann. Probab.}, 39(2):729--778, 2011.

\bibitem{Dirichlet1863}
P.~G.~L. Dirichlet.
\newblock {\em Lectures on number theory}, volume~16 of {\em History of
  Mathematics}.
\newblock American Mathematical Society, Providence, RI; London Mathematical
  Society, London, 1999.
\newblock Supplements by R. Dedekind, Translated from the 1863 German original
  and with an introduction by John Stillwell.

\bibitem{GemanHorowitz1980}
Donald Geman and Joseph Horowitz.
\newblock Occupation densities.
\newblock {\em Ann. Probab.}, 8(1):1--67, 1980.

\bibitem{Hall2003}
Brian~C. Hall.
\newblock {\em Lie groups, {L}ie algebras, and representations}, volume 222 of
  {\em Graduate Texts in Mathematics}.
\newblock Springer-Verlag, New York, 2003.
\newblock An elementary introduction.

\bibitem{KasahKonoOgawa99}
Y.~Kasahara, N.~K\^ono, and T.~Ogawa.
\newblock On tail probability of local times of {G}aussian processes.
\newblock {\em Stochastic Process. Appl.}, 82(1):15--21, 1999.

\bibitem{Kasahara78}
Yuji Kasahara.
\newblock Tauberian theorems of exponential type.
\newblock {\em J. Math. Kyoto Univ.}, 18(2):209--219, 1978.

\bibitem{LiXiao2011}
Yuqiang Li and Yimin Xiao.
\newblock Multivariate operator-self-similar random fields.
\newblock {\em Stochastic Process. Appl.}, 121(6):1178--1200, 2011.

\bibitem{MaejimaMason1994}
Makoto Maejima and J.~David Mason.
\newblock Operator-self-similar stable processes.
\newblock {\em Stochastic Process. Appl.}, 54(1):139--163, 1994.

\bibitem{Pitt78}
Loren~D. Pitt.
\newblock Local times for {G}aussian vector fields.
\newblock {\em Indiana Univ. Math. J.}, 27(2):309--330, 1978.

\bibitem{Tassiulas1997}
L.~Tassiulas.
\newblock Worst case length of nearest neighbor tours for the {E}uclidean
  traveling salesman problem.
\newblock {\em SIAM J. Discrete Math.}, 10(2):171--179, 1997.

\bibitem{Xiao08}
Yimin Xiao.
\newblock Strong local nondeterminism and sample path properties of {G}aussian
  random fields.
\newblock In {\em Asymptotic theory in probability and statistics with
  applications}, volume~2 of {\em Adv. Lect. Math. (ALM)}, pages 136--176. Int.
  Press, Somerville, MA, 2008.

\end{thebibliography}
\def\cprime{$'$}

\end{document}